\documentclass[11pt]{article}
\usepackage[T1]{fontenc}
\usepackage{amsfonts}
\usepackage{amsmath}
\usepackage{amssymb}
\usepackage{amsthm}
\usepackage{bbm}
\usepackage{bm}
\usepackage{mathrsfs}
\usepackage{verbatim}
\usepackage{setspace}
\usepackage{color}
\usepackage{pdfsync}
\usepackage{enumitem}

\theoremstyle{plain}
\newtheorem{theorem}{Theorem}[section]
\newtheorem{proposition}[theorem]{Proposition}
\newtheorem{lemma}[theorem]{Lemma}
\newtheorem{corollary}[theorem]{Corollary}

\theoremstyle{definition}
\newtheorem{definition}[theorem]{Definition}
\newtheorem{remark}[theorem]{Remark}

\theoremstyle{remark}

\renewenvironment{thebibliography}[1]{%
\begin{oldthebibliography}{#1}%
\setlength{\baselineskip}{.9em}
\linespread{1}
\small
\setlength{\parskip}{0.2ex}%
\setlength{\itemsep}{.30em}%
}%
{%
\end{oldthebibliography}%
}

\newcommand{\eps}{\varepsilon}

\newcommand{\F}{\mathbb{F}}
\newcommand{\G}{\mathbb{G}}

\newcommand{\N}{\mathbb{N}}
\newcommand{\Q}{\mathbb{Q}}
\newcommand{\R}{\mathbb{R}}
\renewcommand{\S}{\mathbb{S}}

\newcommand{\cA}{\mathcal{A}}
\newcommand{\cB}{\mathcal{B}}

\newcommand{\cE}{\mathcal{E}}
\newcommand{\cF}{\mathcal{F}}
\newcommand{\cG}{\mathcal{G}}

\newcommand{\cL}{\mathcal{L}}

\newcommand{\cN}{\mathcal{N}}

\newcommand{\cP}{\mathcal{P}}

\newcommand{\fP}{\mathfrak{P}}
\newcommand{\fM}{\mathfrak{M}}

\DeclareMathOperator{\Var}{Var}
\DeclareMathOperator{\conv}{conv}

\newcommand{\as}{\mbox{-a.s.}}

\newcommand{\1}{\mathbf{1}}

\numberwithin{equation}{section}

\usepackage[pdfborder={0 0 0}]{hyperref}
\hypersetup{
  urlcolor = black,
  pdfauthor = {Ariel Neufeld, Marcel Nutz},
  pdfkeywords = {Semimartingale Characteristics;  Semimartingale Property; Doob--Meyer Decomposition},
  pdftitle = {Measurability of Semimartingale Characteristics with Respect to the Probability Law},
  pdfsubject = {Measurability of Semimartingale Characteristics with Respect to the Probability Law},
  pdfpagemode = UseNone
}

\renewcommand{\P}{\mathcal{P}}
\newcommand{\B}{\mathcal{B}}

\begin{document}

\title{\vspace{-0em}
Measurability of Semimartingale Characteristics with Respect to the Probability Law
\date{\today}
\author{
  Ariel Neufeld%
  \thanks{
  Dept.\ of Mathematics, ETH Zurich, \texttt{ariel.neufeld@math.ethz.ch}.
  Financial support by Swiss National Science Foundation Grant PDFMP2-137147/1 is gratefully acknowledged.
  }
  \and
  Marcel Nutz%
  \thanks{
  Dept.\ of Mathematics, Columbia University, New York, \texttt{mnutz@math.columbia.edu}. Financial support by NSF Grant DMS-1208985 is gratefully acknowledged.
  }
 }
}
\maketitle \vspace{-1em}

\begin{abstract}
Given a c\`adl\`ag process $X$ on a filtered measurable space, we construct a version of its semimartingale characteristics which is measurable with respect to the underlying probability law. More precisely, let $\mathfrak{P}_{sem}$ be the set of all probability measures $P$ under which $X$ is a semimartingale. We construct processes $(B^P,C,\nu^P)$ which are jointly measurable in time, space, and the probability law $P$, and are versions of the semimartingale characteristics of $X$ under $P$ for each $P\in\mathfrak{P}_{sem}$. This result gives a general and unifying answer to measurability questions that arise in the context of quasi-sure analysis and stochastic control under the weak formulation.
\end{abstract}

\vspace{.9em}

{\small

}

\section{Introduction}\label{se:intro}

We study the measurability of semimartingale characteristics with respect to the probability law. For the purpose of this introduction, consider the coordinate-mapping process $X$ on the Skorohod space $\Omega=D[0,\infty)$; that is, the set of right-continuous paths with left limits. If $P$ is a law on $\Omega$ such that $X$ is a $P$-semimartingale, we can consider the corresponding triplet $(B^P,C^P,\nu^P)$ of predictable semimartingale characteristics. Roughly speaking, $B^P$ describes the drift, $C^P$ the continuous diffusion, and $\nu^P$ the jumps of $X$. This triplet depends on $P$ and is defined $P$-almost surely; for instance, if $P'$ is equivalent to $P$, the characteristics under $P'$ are in general different from the ones under $P$, whereas if $P$ and $P'$ are singular, it is a priori meaningless to compare the characteristics. In standard situations of stochastic analysis, the characteristics are considered under a fixed probability, or one describes their transformation under an absolutely continuous change of measure as in Girsanov's theorem.

There are, however, numerous applications of stochastic analysis and dynamic programming where we work with a large set $\fP$ of semimartingale laws, often mutually singular.
For instance, when considering a standard stochastic control problem based on a controlled stochastic differential equation, it is useful to recast the problem on Skorohod space by taking $\fP$ to be the set of all laws of solutions of the controlled equation; see e.g.\ \cite{Nutz.11}. This so-called weak formulation of the control problem is advantageous because the Skorohod space has a convenient topological structure; in fact, control problems are often stated directly in this form (cf.\ \cite{ElKaroui.81,Elliott.82} among many others). A similar weak formulation exists in the context of stochastic differential games; here this choice is even more important as the existence of a value may depend on the formulation; see~\cite{PhamZhang.12, Sirbu.13} and the references therein. Or, in the context of a nonlinear expectation $\cE(\cdot)$, the set $\fP$ of all measures $P$ such that $E_P[\cdot]\leq\cE(\cdot)$ plays an important role; see \cite{NutzVanHandel.12, Peng.10, Peng.10icm}. For instance,  the set of all laws of continuous semimartingales whose drift and diffusion coefficients satisfy given bounds is related to $G$-Brownian motion. Other examples where sets of semimartingale laws play a role are path-dependent PDEs~\cite{EkrenTouziZhang.12}, robust superhedging as in~\cite{NeufeldNutz.12, PossamaiRoyerTouzi.13} or nonlinear optimal stopping problems as in~\cite{NutzZhang.12}, to name but a few.
It is well known that the dynamic programming principle is delicate as soon as the regularity of the value function is not known a priori; this is often the case when the reward/cost function is discontinuous or in the presence of state constraints. In this situation, the measurability of the set of controls is crucial to establish the dynamic programming and the measurability of the value function; see \cite{ElKarouiTan.13a, NutzVanHandel.12, Zitkovic.13} for recent developments related to the present paper.

As a guiding example, let us consider the set $\fP$ that occurs in the probabilistic construction of nonlinear L\'evy processes \cite{NeufeldNutz.13b} and which was our initial motivation. The starting point is a collection $\Theta\subseteq \R^d\times \S^d_+\times \cL$, where $\cL$ denotes the set of L\'evy measures; the collection plays the role of a generalized L\'evy triplet since the case of a singleton corresponds to a classical L\'evy process with corresponding triplet. In this application, the set $\fP$ of interest consists of all laws of semimartingales whose differential characteristics take values in $\Theta$, and since a dynamic programming principle is crucial to the theory, we need to establish the measurability of $\fP$. After a moment's reflection, we see that the fundamental question underlying such issues is \emph{the measurability of the characteristics as a function of the law} $P$; indeed, $\fP$ is essentially the preimage of $\Theta$ under the mapping which associates to $P$ the characteristics of $X$ under $P$.
There are of course many situations where a set $\fP$ is specified not only in terms of semimartingale characteristics but with additional conditions whose form is specific to the problem at hand (e.g.\ \cite{NeufeldNutz.12}). However, in view of the fact that intersections of measurable sets are measurable, it makes sense to analyze in general the measurability of the characteristics and check other conditions on a case-by-case basis. Moreover, let us mention that the set $\fP$ often fails to be closed (e.g., because pure jump processes can converge to a continuous diffusion), so that it is indeed natural to examine the measurability directly.

Our main result (Theorem~\ref{th:charactMbl}) states that the set $\fP_{sem}$ of all semimartingale laws is Borel-measurable and that there exists a Borel-measurable map
\[
  \fP_{sem}\times \Omega\times \R_+ \to \R^d\times \S^d_+ \times \overline{\cL},\quad (P,\omega,t)\mapsto (B^P_t(\omega),C_t(\omega),\nu^P(\omega))
\]
such that $(B^P,C,\nu^P)$ are $P$-semimartingale characteristics of $X$ for each $P\in\fP_{sem}$, where $\overline{\cL}$ is the space of L\'evy measures on $\R_+\times \R^d$. A similar result is obtained for the differential characteristics (Theorem~\ref{th:diffCharactMbl}). The second characteristic $C$ can be constructed as a single process not depending on~$P$; roughly speaking, this is possible because two measures under which $X$ has different diffusion are necessarily singular. By contrast, the first and the third characteristic have to depend on $P$ as they are predictable compensators.

Our construction of the characteristics proceeds through versions of the classical results on the structure of semimartingales, such as the Doob--Meyer theorem, with an additional measurable dependence on the law $P$. This situation is somewhat unusual because the objects of interest are probabilistic in nature and at the same time the underlying measure $P$ itself plays the role of the measurable parameter; we are not aware of a similar problem in the literature. The starting point is that for discrete-time processes, the Doob decomposition can be constructed explicitly and of course all adapted processes are semimartingales. Thus, the passage to the continuous-time limit is the main obstacle, just like in the classical theory of semimartingales. A variety of related compactness arguments have emerged over the years; for our purposes, we have found the recent proofs of \cite{BeiglbockSchachermayerVeliyev.11, BeiglbockSchachermayerVeliyev.12} for the Doob--Meyer and the Bichteler--Dellacherie theorem to be particularly useful as they are built around a compactness argument for which we can provide a measurable version.

The remainder of this paper is organized as follows. In Section~\ref{se:mainResults}, we describe the setting and terminology in some detail (mainly because we cannot work with the ``usual assumptions'') and proceed to state the main results. Section~\ref{se:auxResults} contains some auxiliary results, in particular a version of Alaoglu's theorem for~$L^2(P)$ which allows to choose convergent subsequences that depend measurably on~$P$. The measurability of the set of all semimartingale laws is proved in Section~\ref{se:semimartPropertyAndPsem}. In Section~\ref{se:DoobMeyer}, we show that the Doob--Meyer decomposition can be chosen to be measurable with respect to~$P$ and deduce corresponding results for the compensator of a process with integrable variation and the canonical decomposition of a bounded semimartingale. Using these tools, the jointly measurable version of the characteristics is constructed in Section~\ref{se:characteristics}, whereas the corresponding results for the differential characteristics are obtained in the concluding Section~\ref{se:diffCharacteristics}.

\section{Main Results}\label{se:mainResults}

\subsection{Basic Definitions and Notation}

Let $(\Omega,\cF)$ be a measurable space and let $\F=(\cF_t)_{t\geq0}$ be a filtration of sub-$\sigma$-fields of $\cF$. A process $X=(X_t)$ is called right-continuous if \emph{all} its paths are right-continuous. In the presence of a probability measure $P$, we shall say that $X$ is $P$-a.s.\ right-continuous if $P$-almost all paths are right-continuous; the same convention is used for other path properties such as being c\`adl\`ag, of finite variation, etc.

We denote by $\F_{+}:=(\cF_{t+})$ the right-continuous version of $\F$, defined by $\cF_{t+}= \cap_{u>t}\cF_u$. Similarly, the left-continuous version is $\F_-=(\cF_{t-})$. For $t=0$, we use the convention $\cF_{0-}=\cF_{(0+)-}=\{\emptyset,\Omega\}$. As a result, the predictable $\sigma$-field $\cP$ of $\F$ on $\Omega\times\R_+$, generated by the $\F_-$-adapted processes which are left-continuous on $(0,\infty)$, coincides with the predictable $\sigma$-field of $\F_+$; this fact will be used repeatedly without further mention.
Given a probability measure $P$, the augmentation $\F_+^P=(\cF^P_{t+})$ of $\F_+$, also called the usual augmentation of $\F$, is obtained by adjoining all $P$-nullsets of $(\Omega,\cF)$ to $\cF_{t+}$ for all $t$, including $t=0-$. The corresponding predictable $\sigma$-field will be denoted by $\cP^P$.

Finally, $\fP(\Omega)$ is the set of all probability measures on $(\Omega,\cF)$. In most of this paper, $\Omega$ will be a separable metric space and $\cF$ its Borel $\sigma$-field. In this case,
$\fP(\Omega)$ is a separable metric space for the weak convergence of probability measures and its Borel $\sigma$-field $\cB(\fP(\Omega))$ coincides with the one generated by the maps $P\mapsto P(A)$, $A\in\cF$. Unless otherwise mentioned, any metric space is equipped with its Borel $\sigma$-field. Similarly, product spaces are always equipped with their product $\sigma$-fields and measurability then refers to joint measurability.

It will be convenient to define the integral of any (appropriately measurable) function $f$ taking values in the extended real line $\overline{\R}=[-\infty,\infty]$, regardless of its integrability. For instance, the expectation under a probability measure $P$ is defined by $E^P[f]:= E^P[f^+]- E^P[f^-]$; here and everywhere else, the convention
\[
  \infty - \infty = -\infty
\]
is used. Similarly, conditional expectations are also defined for $\overline{\R}$-valued functions.

\begin{definition}\label{def:semimartingale}
  Let $(\Omega,\cG,\G,P)$ be a filtered probability space. A  $\mathbb{G}$-adapted stochastic process $X: \Omega\times \R_+ \to \R^d$ with c\`adl\`ag paths is a $P$-$\G$-semimartingale if there exist right-continuous, $\mathbb{G}$-adapted processes $M$ and $A$ with $M_0=A_0=0$ such that $M$ is a $P$-$\G$-local martingale, $A$ has paths of (locally) finite variation $P$-a.s., and
  \begin{equation*}
    X= X_0 + M + A \ \ P\mbox{-a.s.}
  \end{equation*}
\end{definition}

The dimension $d\in\N$ is fixed throughout. Fix also a truncation function $h: \R^d\to \R^d$; that is, a bounded measurable function such that $h(x)=x$ in a neighborhood of the origin. The characteristics of a semimartingale $X$ on $(\Omega,\cG,\G,P)$ are a triplet $(B,C,\nu)$ of processes defined as follows. First, consider the c\`adl\`ag process
\begin{equation*}
  \widetilde{X}_t:= X_t-X_0-\sum_{0\leq s \leq t} \big(\Delta X_s- h(\Delta X_s)\big),
\end{equation*}
which has bounded jumps. This process has a ($P$-a.s.\ unique) canonical decomposition $\widetilde{X}=M'+B'$, where $M'$ and $B'$ have the same properties as the processes in Definition~\ref{def:semimartingale}, but in addition $B'$ is predictable. (See \cite[Theorem~7.2.6, p.\,160]{vonWeizsackerWinkler.90} for the existence of the canonical decomposition in a general filtration.) Moreover, let $\mu^X$ be the integer-valued random measure associated with the jumps of $X$,
\[
  \mu^X(\omega,dt,dx)=\sum_{s\geq0} \1_{\{\Delta X_s(\omega)\neq0\}} \1_{(s,\Delta X_s(\omega))}(dt,dx).
\]
Processes $(B,C,\nu)$ with values in $\R^d$, $\R^{d\times d}$, and the set of measures on $\R_+\times\R^d$, respectively, will be called characteristics of $X$ (relative to $h$) if $B=B'$ $P$-a.s., $C$ equals the predictable covariation process of the continuous local martingale part of $M'$ $P$-a.s., and $\nu$ equals the predictable compensator of $\mu^X$  $P$-a.s. All these notions are relative to the given filtration $\G$ which, in the sequel, will be either the basic filtration $\F$, its right-continuous version $\F_+$, or its usual augmentation $\F^P_+$. Our first observation is that the characteristics do not depend on this choice.

\begin{proposition}\label{prop:SemFunab}
  Let $X$ be a c\`adl\`ag, $\R^d$-valued, $\F$-adapted process on a filtered probability space $(\Omega,\cF,\F,P)$. The following are equivalent:
  \begin{enumerate}[topsep=3pt, partopsep=0pt, itemsep=2pt,parsep=2pt]
    \item $X$ is an $\F$-semimartingale,
    \item $X$ is an $\F_+$-semimartingale,
    \item $X$ is an $\F^P_+$-semimartingale.
  \end{enumerate}
  Moreover, the semimartingale characteristics associated with these filtrations are the same.
\end{proposition}

The proof is stated in Section~\ref{se:semimartPropertyAndPsem}. In order to study the measurability of the third characteristic $\nu$, we introduce a $\sigma$-field on the set of L\'evy measures; namely, we shall use the Borel $\sigma$-field associated with a natural metric that we define next. Given a metric space $\Omega'$, let $\fM(\Omega')$ denote the set of all (nonnegative) measures on $(\Omega',\cB(\Omega'))$. We introduce the set of L\'evy measures on $\R^d$,
\begin{equation*}
  \cL= \bigg\{ \nu\in \fM(\R^d) \, \bigg| \,  \int_{\R^d} |x|^2 \wedge 1 \, \nu(dx) <\infty \ \mbox{and} \ \nu(\{0\})=0  \bigg\},
\end{equation*}
as well as their analogues on $\R_+\times \R^d$,
\begin{multline}\label{eq:defOverlineL}
  \overline{\cL}= \bigg\{ \nu\in \fM(\R_+\times \R^d) \, \bigg| \, \int_0^N \int_{\R^d} |x|^2 \wedge 1 \, \nu(dt,dx) <\infty \;\forall\, N \in \N,  \\
  \nu(\{0\} \times \R^d)=\nu(\R_+\times\{0\})=0 \bigg\}.
\end{multline}
The space $\fM^f(\R^d)$ of all finite measures on $\R^d$ is a separable metric space under a metric $d_{\fM^f(\R^d)}$ which induces the weak convergence relative to $C_b(\R^d)$; cf.\ \cite[Theorem~8.9.4, p.\,213]{Bogachev.07volII}; this topology is the natural extension of the more customary weak convergence of probability measures. With any $\mu\in\cL$, we can associate a finite measure
\begin{equation*}
  A \mapsto \int_A |x|^2 \wedge 1 \, \mu(dx), \ \ \ A \in \mathcal{B}(\R^d),
\end{equation*}
denoted by $|x|^2 \wedge 1.\mu$ for brevity. We can then define a metric $d_{\cL}$ on $\cL$ via
\begin{equation*}
  d_{\cL}(\mu,\nu)= d_{\fM^f(\R^d)}\big(|x|^2 \wedge 1.\mu, |x|^2 \wedge 1.\nu\big), \quad \mu,\nu \in \cL.
\end{equation*}
We proceed similarly with $\overline{\cL}$. First, given $N>0$, let $\overline{\cL}_N$ be the restriction of $\overline{\cL}$ to $[0,N]\times \R^d$. For any $\mu\in\overline{\cL}_N$, let $|x|^2 \wedge 1.\mu$ be the finite measure
\begin{equation*}
  A \mapsto \int_A |x|^2 \wedge 1 \, \mu(dt,dx), \ \ \ A \in \mathcal{B}([0,N]\times \R^d);
\end{equation*}
then we can again define a metric
\begin{equation*}
  d_{\overline{\cL}_N}(\mu,\nu)= d_{\fM^f([0,N]\times \R^d)}\big(|x|^2 \wedge 1.\mu, |x|^2 \wedge 1.\nu\big), \quad \mu,\nu \in \overline{\cL}_N.
\end{equation*}
Finally, we can metrize $\overline{\cL}$ by
\[
  d_{\overline{\cL}}(\mu,\nu)=\sum_{N\in\N} 2^{-N} \big(1\wedge d_{\overline{\cL}_N}(\mu,\nu) \big), \quad \mu,\nu \in \overline{\cL}.
\]

\begin{lemma}\label{le:metric}
  The pairs $(\cL,d_{\cL})$, $(\overline{\cL}_N,d_{\overline{\cL}_N})$, $(\overline{\cL},d_{\overline{\cL}})$ are separable metric spaces.
\end{lemma}

This is proved by reducing to the properties of $\fM^f$; we omit the details. The above metrics define the Borel structures $\cB(\cL)$, $\cB(\overline{\cL}_N)$ and $\cB(\overline{\cL})$. Alternatively, we could have defined the $\sigma$-fields through the following result, which will be useful later on.

\begin{lemma}\label{le:kernlevy}
  Let $(Y,\mathcal{Y})$ be a measurable space and consider a function $\kappa:\, Y\to \overline{\cL}$, $y\mapsto \kappa(y,dt,dx)$. The following are equivalent:
  \begin{enumerate}
    \item $\kappa:\, (Y,\mathcal{Y})\to (\overline{\cL},\cB(\overline{\cL}))$ is measurable,
    \item for all measurable functions $f: \R_+\times\R^d\to \R$,
    \[
      (Y,\mathcal{Y})\to (\R,\cB(\R)),\quad  y \mapsto \int_0^\infty \int_{\R^d} (|x|^2 \wedge 1)f(t,x)\, \kappa(y,dt,dx)
    \]
    is measurable.
  \end{enumerate}
  Corresponding assertions hold for $\cL$ and $\overline{\cL}_N$.
\end{lemma}

\begin{proof}
  A similar result is standard, for instance, for the set of probability measures on a Polish space; cf.\ \cite[Proposition~7.25, p.\,133]{BertsekasShreve.78}. The arguments in this reference can be adapted to the space $\overline{\cL}_N$ by using the facts stated in \cite[Chapter~8]{Bogachev.07volII}. Then, one can lift the result to the space $\overline{\cL}$. We omit the details.
\end{proof}

\subsection{Main Results}

We can now state our main result, the existence of a jointly measurable version $(P,\omega,t)\mapsto (B^P_t(\omega),C_t(\omega),\nu^P(\omega))$ of the characteristics of a process $X$ under a family of measures $P$. Here the second characteristic $C$ is a single process not depending on $P$; roughly speaking, this is possible because two measures under which $X$ has different diffusion are necessarily singular. By contrast, the first and the third characteristic have to depend on $P$ in all nontrivial cases: in general, two equivalent measures will lead to different drifts and compensators, so that the families $(B^P)_P$ and $(\nu^P)_P$ are not consistent with respect to $P$ and cannot be aggregated into single processes. We write $\S^d_+$ for the set of symmetric nonnegative definite  $d\times d$-matrices.

\begin{theorem}\label{th:charactMbl}
  Let $X$ be a c\`adl\`ag, $\F$-adapted, $\R^d$-valued process on a filtered measurable space $(\Omega,\cF,\F)$, where $\Omega$ is a separable metric space, $\cF=\cB(\Omega)$ and each $\sigma$-field $\cF_t$ of the filtration $\F=(\cF_t)_{t\geq0}$ is separable. Then the set
  \[
    \fP_{sem}=\{P\in \fP(\Omega)\,|\, \mbox{$X$ is a semimartingale on }(\Omega,\cF,\F,P)\}\subseteq \fP(\Omega)
  \]
  is Borel-measurable and there exists a Borel-measurable map
  \[
    \fP_{sem}\times \Omega\times \R_+ \to \R^d\times \S^d_+ \times \overline{\cL},\quad (P,\omega,t)\mapsto (B^P_t(\omega),C_t(\omega),\nu^P(\omega))
  \]
  such that for each $P\in\fP_{sem}$,
  \begin{enumerate}
    \item $(B^P,C,\nu^P)$ are $P$-semimartingale characteristics of $X$,
    \item $B^P$ is $\F_{+}$-adapted, $\F^P_{+}$-predictable and has right-continuous, $P$-a.s.\ finite variation paths,
    \item $C$ is $\F$-predictable and has $P$-a.s.\ continuous, increasing paths\footnote{Alternately, one can construct $C$ such that all paths are continuous and increasing, at the expense of being predictable in a slightly larger filtration. See Proposition~\ref{prop:quadvarunabP}.} in $\S^d_+$,
    \item $\nu^P$ is an $\F^P_{+}$-predictable random measure on $\R_+\times \R^d$.
  \end{enumerate}
  Moreover, there exists a decomposition
  \begin{equation*}
     \nu^P(\cdot,dt,dx) = K^P(\cdot,t,dx)\,dA^P_t\quad P\as,
  \end{equation*}
  where
  \begin{enumerate}[topsep=3pt, partopsep=0pt, itemsep=2pt,parsep=2pt]
  \addtocounter{enumi}{4}
    \item $(P,\omega,t)\mapsto A^P_t(\omega)$ is Borel-measurable and for all $P \in \fP_{sem}$, $A^P$ is an $\F_{+}$-adapted, $\F^P_+$-predictable, $P$-integrable process with right-continuous and $P$-a.s.\ increasing paths,
    \item $(P,\omega,t)\mapsto K^P(\omega,t,dx)$ is a kernel on $(\R^d,\B(\R^d))$ given $(\fP_{sem}\times\Omega\times\R_+, \B(\fP_{sem})\otimes \cF\otimes \cB(\R_+))$ and for all $P \in \fP_{sem}$, $(\omega,t)\mapsto K^P(\omega,t,dx)$ is a kernel on $(\R^d,\B(\R^d))$ given $(\Omega\times\R_+, \P^P)$.
  \end{enumerate}
\end{theorem}
The measurability of $\fP_{sem}$ is proved in Section~\ref{se:semimartPropertyAndPsem}, whereas the characteristics are constructed in Section~\ref{se:characteristics}. We remark that the conditions of the theorem are satisfied in particular when $X$ is the coordinate-mapping process on Skorohod space and $\F$ is the filtration generated by $X$. This is by far the most important example---the slightly more general situation in the theorem does not cause additional work.

Of course, we are particularly interested in measures $P$ such that the characteristics are absolutely continuous with respect to the Lebesgue measure $dt$ on $\R_+$; that is, the set
\[
  \fP^{ac}_{sem}=\big\{P\in \fP_{sem}\,\big|\, \mbox{$(B^P,C,\nu^P)\ll dt$, $P$-a.s.}\big\}.
\]
(Absolute continuity does not depend on the choice of the truncation function~$h$; cf.\ \cite[Proposition~2.24, p.\,81]{JacodShiryaev.03}.) Given a triplet of absolutely continuous characteristics, the corresponding derivatives (defined $dt$-a.e.) are called the differential characteristics of $X$ and denoted by $(b^P,c,F^P)$.

\begin{theorem}\label{th:diffCharactMbl}
  Let $X$ and $(\Omega,\cF,\F)$ be as in Theorem~\ref{th:charactMbl}. Then the set
  \[
    \fP^{ac}_{sem}=\big\{P\in \fP_{sem}\,\big|\, \mbox{$(B^P,C,\nu^P)\ll dt$, $P$-a.s.}\big\}
  \]
  is Borel-measurable and there exists a Borel-measurable map
  \[
    \fP^{ac}_{sem}\times \Omega\times \R_+ \to \R^d\times \S^d_+ \times \cL,\quad (P,\omega,t)\mapsto (b^P_t(\omega),c_t(\omega),F^P_{\omega,t})
  \]
  such that for each $P\in\fP^{ac}_{sem}$,
  \begin{enumerate}
    \item $(b^P,c,F^P)$ are $P$-differential characteristics of $X$,
    \item $b^P$ is $\F$-predictable,
    \item $c$ is $\F$-predictable,
    \item $(\omega,t)\mapsto F^P_{\omega,t}(dx)$ is a kernel on $(\R^d,\B(\R^d))$ given $(\Omega\times\R_+, \P)$.
  \end{enumerate}
\end{theorem}

In applications, we are interested in constraining the set $\fP^{ac}_{sem}$ via the values of the differential characteristics. Given a collection $\Theta\subseteq \R^d \times \S^d_+\times\cL$ of L\'evy triplets, we let
\[
  \fP^{ac}_{sem}(\Theta)=\big\{P\in \fP^{ac}_{sem}\,\big|\, (b^P,c,F^P)\in \Theta, P\otimes dt\mbox{-a.e.}\big\}.
\]

\begin{corollary}\label{co:PThetaMbl}
  Let $X$ and $(\Omega,\cF,\F)$ be as in Theorem~\ref{th:charactMbl}. Then $\fP^{ac}_{sem}(\Theta)$ is Borel-measurable whenever $\Theta\subseteq \R^d \times \S^d_+\times\cL$ is Borel-measurable.
\end{corollary}

The proofs for Theorem~\ref{th:diffCharactMbl} and Corollary~\ref{co:PThetaMbl} are stated in Section~\ref{se:diffCharacteristics}.

\begin{remark}
  The arguments in the subsequent sections yield similar results when $X$ is $\F_+$-adapted (instead of $\F$-adapted), or if $X$ is replaced by an appropriately measurable family $(X^P)_P$ as in Proposition~\ref{prop:Doobmeyer} below---we have formulated the main results in the setting which is most appropriate for the applications we have in mind.
\end{remark}

\section{Auxiliary Results}\label{se:auxResults}

This section is a potpourri of tools that will be used repeatedly later on; they mainly concern the possibility of choosing $L^1(P)$-convergent subsequences and limits in a measurable way (with respect to~$P$). Another useful result concerns right-continuous modifications of processes.

Throughout this section, we place ourselves in the setting of Theorem~\ref{th:charactMbl}; that is, $\Omega$ is a separable metric space, $\cF=\cB(\Omega)$ and $\F=(\cF_t)$ is a filtration such that $\cF_t$ is separable for all $t\geq0$. Moreover, we fix a measurable set $\fP\subseteq \fP(\Omega)$; recall that $\fP(\Omega)$ carries the Borel structure induced be the weak convergence. (The results of this section also hold for a general measurable space $(\Omega,\cF)$ if $\fP(\Omega)$ is instead endowed with the $\sigma$-field generated by the maps $P\mapsto P(A)$, $A\in\cF$.)

As $P$ plays the role of a measurable parameter, it is sometimes useful to consider the filtered measurable space
\begin{equation}\label{eq:defHatOmega}
  \big(\widehat{\Omega},\widehat{\cF}\big):= \big(\fP\times \Omega, \cB(\fP)\otimes \cF\big), \quad \widehat{\F}=(\widehat{\cF}_t)_{t\geq0},\quad \widehat{\cF}_t:=\cB(\fP)\otimes \cF_t
\end{equation}
and its right-continuous filtration $\widehat{\F}_+$; a few facts can be obtained simply by applying standard results in this extended space.

\begin{lemma}\label{le:BorelcondexpPP}
  Let $t\geq0$ and let $f:\widehat{\Omega} \to \overline{\R}$ be measurable. Then the function $\fP\to \overline{\R}$, $P\mapsto E^P[f(P,\cdot)]$ is measurable. Moreover, there exist versions of the conditional expectations
  $E^P[f(P,\cdot) \, | \,\cF_t]$ and $E^P[f(P,\cdot) \, | \,\cF_{t+}]$ such that
  \[
    \widehat{\Omega} \to \overline{\R}, \quad (P,\omega)\mapsto E^P[f(P,\cdot) \, | \,\cF_t](\omega), \quad (P,\omega)\mapsto E^P[f(P,\cdot) \, | \,\cF_{t+}](\omega)
  \]
  are measurable with respect to $\widehat{\cF}_t$ and $\widehat{\cF}_{t+}$, respectively, while for fixed $P\in\fP$,
  \[
    \Omega \to \overline{\R}, \quad \omega\mapsto E^P[f(P,\cdot) \, | \,\cF_t](\omega), \quad \omega\mapsto E^P[f(P,\cdot) \, | \,\cF_{t+}](\omega)
  \]
  are measurable with respect to $\cF_t$ and $\cF_{t+}$, respectively.
\end{lemma}

\begin{proof}
  It suffices to consider the case where $f$ is bounded. We first show that $P\mapsto E^P[f(P,\cdot)]$ is measurable. By a monotone class argument, it suffices to consider a function $f$ of the form
  $f(P,\omega)=g(P)h(\omega)$, where $g$ and $h$ are measurable. In this case, $P\mapsto E^P[f(P,\cdot)]= g(P)\,E^P[h]$, and $P\mapsto E^P[h]$ is measurable due to \cite[Proposition~7.25, p.\,133]{BertsekasShreve.78}.

  The construction of the conditional expectation follows the usual scheme. Fix $t\geq 0$, let $(A_n)_{n\in \N}$ be a sequence generating $\cF_t$ and let $(A_n^m)_m$ be a finite partition generating $\cA_n:=\sigma(A_1,\dots,A_n)$. Using the supermartingale convergence theorem as in \cite[V.56, p.\,50]{DellacherieMeyer.82} and the convention $0/0=0$, we can define a version of the conditional expectation given $\cF_t$ by
  \begin{equation*}
    E^P[f(P,\cdot)\, | \, \cF_t ]:= \limsup_{n \to \infty} \sum_{m} \frac{E^P[f(P,\cdot) \,\1_{A_n^m}]}{P[A_n^m]} \,\1_{A_n^m}.
  \end{equation*}
  In view of the first part, this function is $\widehat{\cF}_t$-measurable, and $\cF_t$-measurable for fixed $P$. Finally, using the backward martingale convergence theorem,
  \begin{equation*}
    E^P[f(P,\cdot)\, | \, \cF_{t+} ] :=\limsup_{n \to \infty} E^P[f(P,\cdot)\, | \, \cF_{t+1/n} ]
  \end{equation*}
  is a version of the conditional expectation given $\cF_{t+}$ having the desired properties.
\end{proof}

In what follows, we shall always use the measurable versions of the conditional expectations as in Lemma~\ref{le:BorelcondexpPP}.

\begin{lemma}\label{le:L1messbar}
  Let $f^n: \widehat{\Omega} \to \overline{\R}^d$ be measurable functions such that $f^n(P,\cdot)$ is a convergent sequence in $L^1(P)$ for every $P\in\fP$.
  There exists a measurable function $f:\widehat{\Omega} \to \overline{\R}^d$ such that $f(P,\cdot)=\lim_n f^n(P,\cdot)$ in $L^1(P)$ for every $P\in\fP$. Moreover, there exists an increasing sequence $(n_k^P)_k\subseteq \N$ such that $P\mapsto n_k^P$ is measurable and $\lim_k f^{n_k^P}(P,\cdot)=f(P,\cdot)$ $P$-a.s.\ for all $P\in\fP$.
\end{lemma}

\begin{proof}
  For $P\in\fP$, let $n_0^P:=1$ and define recursively
  \begin{align*}
  \tilde{n}_k^P &:= \min\big\{ n \in \N \, \big| \,  \Vert f^u(P,\cdot)- f^v(P,\cdot) \Vert_{L^1(P)} \leq 2^{-k} \ \mbox{for all} \ u,v \geq n \big\},\\
  n_k^P &:= \max\big\{\tilde{n}_k^P,n_{k-1}^P +1\big\}.
  \end{align*}
  It follows from Lemma~\ref{le:BorelcondexpPP} that $P\mapsto n^P_k$ is measurable, and so the composition $(P,\omega)\mapsto f^{n_k^P}(P,\omega)$ is again measurable.
  Moreover, we have
  \[
    \sum_{k\geq0} \Vert f^{n^P_{k+1}}(P,\cdot)-f^{n^P_k}(P,\cdot) \Vert_{L^1(P)}<\infty,\quad P\in\fP
  \]
  by construction, which implies that $(f^{n^P_k}(P,\cdot))_{k \in \N}$ converges $P$-a.s.
  Thus, we can set (componentwise)
  \begin{equation*}
    f(P,\omega):=\limsup_{k \to \infty} f^{n_k^P}(P,\omega)
  \end{equation*}
  to obtain a jointly measurable limit.
\end{proof}

The next result is basically a variant of Alaoglu's theorem in $L^2$ (or the Dunford--Pettis theorem in $L^1$, or Komlos' lemma) which yields measurability with respect to the underlying measure. It will be crucial to obtain measurable versions of the compactness arguments of semimartingale theory in the later sections. We denote by $\conv A$ the convex hull of a set $A\subseteq \R^d$.

\begin{proposition}\label{pr:mblKomlos}
  (i) Let $f^n:\fP \times \Omega \to \R^d$ be a sequence of measurable functions such that
  \begin{equation}\label{eq:L2KomlosAssump}
    \sup_{n \in \N} \Vert f^n(P,\cdot) \Vert_{L^2(P)}<\infty,\quad P \in \fP.
  \end{equation}
  Then there exist measurable functions
  $P\mapsto N_n^P \in \{n,n+1,\dots\}$ and $P\mapsto \lambda_i^{P,n}\in [0,1]$
  satisfying $\sum_{i=n}^{N^P_n} \lambda_i^{P,n}=1$ and  $\lambda^{P,n}_i=0$ for $i \notin \{n,...,N^P_n\}$ such that
  \[
    (P,\omega)\mapsto g^{P,n}(\omega):= \sum_{i=n}^{N^P_n} \lambda^{P,n}_i\,f^{i}(P,\omega) \ \in \conv\{f^n(P,\omega),f^{n+1}(P,\omega),...\}
  \]
  is measurable and $(g^{P,n})_{n\in\N}$ converges in $L^2(P)$ for all $P\in\fP$.\vspace{.5em}

  (ii) For each $m \in \N$,  let $(f_m^n)_{n \in \N}$ be a sequence as in~(i). Then there exist $N_n^P$ and $\lambda^{P,n}_i$ as in (i) such that
  \begin{equation*}
    (P,\omega)\mapsto g^{P,n}_m(\omega):= \sum_{i=n}^{N^P_n} \lambda^{P,n}_i\,f^i_m(P,\omega) \ \in \conv\{f^n_m(P,\omega),f^{n+1}_m(P,\omega),...\}
  \end{equation*}
  is measurable and $(g^{P,n}_m)_{n\in\N}$ converges in $L^2(P)$ for all $P\in\fP$ and $m\in\N$.\vspace{.5em}

  (iii) Let $f^n:\fP \times \Omega \to \R^d$ be measurable functions such that
  \[
    \{f^n(P,\cdot)\}_{n\in\N}\subseteq L^1(P)\quad\mbox{is uniformly integrable},\quad P\in\fP.
  \]
  Then the assertion of (i) holds with convergence in $L^1(P)$ instead of $L^2(P)$.
\end{proposition}

\begin{proof}
  (i) For $n\in\N$, consider the sets
  \[
    G^{P,n} = \conv \{f^n(P,\cdot),f^{n+1}(P,\cdot),...\},\quad P\in\fP.
  \]
  Moreover, for $k \in \N$, let $\Lambda^n_k$ be the (finite) set of all $\lambda=(\lambda_1,\lambda_2,\dots)\in[0,1]^\N$ such that $\sum_{i} \lambda_i =1$,
  \[
     \lambda_i= \frac{a_i}{b_i} \quad \mbox{ for some } a_i\in \{0,1,\dots,b_i\},\quad b_i\in \{1,2,\dots, k\}
  \]
  and $\lambda_i=0$ for $i\notin \{n,\dots,n+k\}$. Thus,
  \[
    g^P(\lambda):= \sum_{i\geq1} \lambda_i\,f^{i}(P,\cdot) \in G^{P,n}
  \]
  for all $\lambda\in \Lambda^n_k$. Let
  \[
   \alpha^{P,n}_k = \min\big\{ \Vert g^P(\lambda) \Vert_{L^2(P)} \, \big| \, \lambda\in \Lambda^n_k\big\},\quad
   \alpha^{P,n} = \inf\big\{\Vert g \Vert_{L^2(P)} \, \big| \, g \in G^{P,n} \big\}
  \]
  and $\alpha^P =\lim_{n} \alpha^{P,n}$; note that $(\alpha^{P,n})_n$ is increasing. We observe that any sequence $g^{P,n}\in G^{P,n}$ such that $\|g^{P,n}\|_{L^2(P)}\leq \alpha^{P,n} + 1/n$ is a Cauchy sequence in $L^2(P)$. Indeed, if $\eps>0$ is given and $n$ is large,
  then $\|(g^{P,k}+g^{P,l})/2\|_{L^2(P)}\geq \alpha^P -\eps $ for all $k,l\geq n$, which by the parallelogram identity yields that
  \[
    \|g^{P,k}-g^{P,l}\|^2_{L^2(P)} \leq 4(\alpha^{P,n} + 1/n)^2-4(\alpha^P-\eps)^2.
  \]
  As $\alpha^{P,n}$ tends to $\alpha^P$, this shows the Cauchy property.
  To select such a sequence in a measurable way, we first observe that $(\alpha^{P,n}_k)_k$ decreases to $\alpha^{P,n}$, due to~\eqref{eq:L2KomlosAssump}. Thus,
  \begin{equation*}
    k^{P,n}:=\min \big\{ k \in \N \, \big| \, |\alpha_k^{P,n}-\alpha^{P,n}|\leq 1/n\big\}
  \end{equation*}
  is well defined and finite.
  As a consequence of Lemma~\ref{le:BorelcondexpPP}, $P\mapsto (\alpha_k^{P,n},\alpha^{P,n})$ is measurable, and this implies that $P\mapsto k^{P,n}$ is measurable. Applying a selection theorem in the Polish space $[0,1]^\N$ (e.g., \cite[Theorem~18.13, p.\,600]{AliprantisBorder.06}), we can find for each $n$ a measurable minimizer $P\mapsto \widehat{\lambda}^{P,n}$ in the (finite) set $\Lambda^n_{k^{P,n}}$ such that
  \[
    \Vert g^P(\widehat{\lambda}^{P,n}) \Vert_{L^2(P)} = \alpha^{P,n}_{k^{P,n}} \equiv \min\big\{ \Vert g^P(\lambda) \Vert_{L^2(P)} \, \big| \, \lambda\in \Lambda^n_{k^{P,n}}\big\}.
  \]
  According to the above, $g^P(\widehat{\lambda}^{P,n})$ is Cauchy and so the result follows by setting $N^P_n=n + k^{P,n}$.

  (ii) This assertion follows from~(i) by a standard ``diagonal argument.''

  (iii) For $m,n \in \N$, define the function $f_m^n:\fP \times \Omega \to \R^d$ by
  \[
    f_m^n(P,\omega):=f^n(P,\omega)\,\1_{\{|f^n(P,\omega)|\leq m\}}.
  \]
  Then $\sup_{n \in \N} \Vert f^n_m(P,\cdot) \Vert_{L^2(P)}<\infty$ for each $m$. Thus, for each $m$, (ii) yields an $L^2(P)$-convergent sequence
  \begin{equation*}
    g^{P,n}_m = \sum_{i=n}^{N^P_n} \lambda^{P,n}_i\,f^{i}_m(P,\cdot)
  \end{equation*}
  with suitably measurable coefficients. We use the latter to define
  \begin{equation*}
  g^{P,n}:= \sum_{i=n}^{N^P_n} \lambda^{P,n}_i\,f^{i}(P,\cdot).
  \end{equation*}
  By the assumed uniform integrability, we have
  \[
    \lim_{m\to\infty}\sup_{n\geq1 }\|f_m^n(P,\cdot)-f^n(P,\cdot)\|_{L^1(P)}=0,\quad P\in\fP;
  \]
  thus, the Cauchy property of $(g^{P,n})_n$ in $L^1(P)$ follows from the corresponding property of the sequences $(g^{P,n}_m)_n$.
\end{proof}

The last two lemmas in this section are observations about the measurability of processes and certain right-continuous modifications.

\begin{lemma}\label{le:productm}
  Let $f:\fP\times \Omega \times \R_+ \to \overline{\R}$ be such that $f(\cdot,\cdot,t)$  is $\widehat{\cF}_t$-measurable for all $t$ and $f(P,\omega,\cdot)$ is right-continuous for all $(P,\omega)$. Then $f$ is measurable and $f|_{\fP\times \Omega \times[0,t]}$ is $\widehat{\cF}_t \otimes \mathcal{B}([0,t])$-measurable for all $t \in \R_+$.

  The same assertion holds if $\widehat{\cF}_t$ is replaced by $\widehat{\cF}_{t+}$ throughout.
\end{lemma}

\begin{proof}
  This is simply the standard fact that a right-continuous, adapted process is progressively measurable, applied on the extended space $\widehat{\Omega}$.
\end{proof}

Finally, we state a variant on a regularization for processes in right-continuous but non-complete filtrations. As usual, the price to pay for the lack of completion is that the resulting paths are not c\`adl\`ag in general.

\begin{lemma}\label{le:cadlagversionm}
  Let $f:\fP\times \Omega \times \R_+ \to\R$ be such that $f(\cdot,\cdot,t)$ is $\widehat{\cF}_{t+}$-measurable for all $t$. There exists a measurable function $\bar{f}:\fP \times \Omega \times \R_+ \to \R$ such that $\bar{f}$ is $\widehat{\F}_+$-optional, $\bar{f}(P,\omega,\cdot)$ is right-continuous for all $(P,\omega)$, and for any $P \in \fP$ such that $f(P,\cdot,\cdot)$ is an $\F_+$-adapted $P$-$\F_+$-supermartingale with right-continuous expectation $t \mapsto E^P[f(P,\cdot,t)]$, the process $\bar{f}(P,\cdot,\cdot)$ is an $\F_+$-adapted $P$-modification of $f(P,\cdot,\cdot)$ and in particular a $P$-$\F_+$-supermartingale.
\end{lemma}

\begin{proof}
  Let $D$ be a countable dense subset of $\R_+$. For any $a<b \in \R$ and $t \in \R_+$, denote by $M_a^b(D\cap [0,t],P,\omega)$ the number of upcrossings of the restricted path $f(P,\omega,\cdot)|_{D\cap [0,t]}$ over the interval $[a,b]$. Moreover, let
  \begin{align*}
  \tau_a^b(P,\omega)&=\inf\big\{t \in   \Q_+ \, \big| \, M_a^b(D\cap [0,t],P,\omega)=\infty \big\},\\
  \sigma(P,\omega)&= \inf \big\{ t \in  \Q_+ \, \big| \, \sup_{s \leq t,\, s \in D} |f(P,\omega,s)|=\infty \big\},\\
  \rho(P,\omega)&= \sigma(P,\omega) \wedge \inf_{a<b \in \Q} \tau_a^b(P,\omega)
  \end{align*}
  and define the function $\bar{f}$ by
  \begin{equation*}
  \bar{f}(P,\omega,t):= \bigg(\limsup_{s \in D,\, s\downarrow t} f(P,\omega,s)\bigg)\,\1_{\{t< \rho(P,\omega)\}}.
  \end{equation*}
  Using the arguments in the proof of \cite[Remark~VI.5, p.\,70]{DellacherieMeyer.82}, we can verify that $\bar{f}$ has the desired properties.
\end{proof}

\section{Semimartingale Property and $\fP_{sem}$}\label{se:semimartPropertyAndPsem}

In this section, we prove Proposition~\ref{prop:SemFunab} and the measurability of $\fP_{sem}$.

\begin{proof}[Proof of Proposition~\ref{prop:SemFunab}]
  Let $X$ be a c\`adl\`ag, $\F$-adapted process on a filtered probability space $(\Omega,\cF,\F,P)$. We begin with the equivalence of
  \begin{enumerate}[topsep=7pt, partopsep=0pt, itemsep=2pt,parsep=2pt]
    \item $X$ is an $\F$-semimartingale,
    \item $X$ is an $\F_+$-semimartingale,
    \item $X$ is an $\F^P_+$-semimartingale.
  \end{enumerate}
  To see that (i) implies (iii), let $X= X_0 +M+A$ be an $\mathbb{F}$-semimartingale, where $M$ is a right-continuous $\mathbb{F}$-local martingale and $A$ is a right-continuous $\mathbb{F}$-adapted process with paths of $P$-a.s.\ finite variation. The same decomposition is admissible in $\mathbb{F}^P_{+}$; to see this, note that any right-continuous $\mathbb{F}$-martingale $N$ is also an $\mathbb{F}^P_{+}$-martingale: by the backward martingale convergence theorem,  $N_s=E^P[N_t|\cF_s]$ for $s\leq t$ implies
  \[
    N_s=N_{s+}=E^P[N_t|\cF_{s+}]=E^P[N_t|\cF_{s+}^P]\quad P\as, \quad s\leq t.
  \]

  Next, we show that (iii) implies (i). We observe that the process
  \[
    \widetilde{X}_t:=X_t- X_0 -\sum_{0\leq s\leq t} \Delta X_s \, \1_{\{|\Delta X_s|>1\}},
  \]
  is a semimartingale if and only if $X$ is. Thus, we may assume that $X_0=0$ and that $X$ has jumps bounded by one. In particular, $X$ then has a canonical decomposition $X= M+ B$, where $M$ is a right-continuous $\mathbb{F}_{+}^{P}$-local martingale and $B$ is a right-continuous $\mathbb{F}_{+}^{P}$-predictable process of finite variation. We can decompose the latter into a difference
  $B=B^1-B^2$ of increasing, right-continuous $\F^P_{+}$-predictable processes. By \cite[Lemma~6.5.10, p.\,143]{vonWeizsackerWinkler.90}, there exist right-continuous, $P$-a.s.\ increasing and $\F$-predictable processes $\hat{B}^1$ and $\hat{B}^2$ which are indistinguishable from $B^1$ and $B^2$, respectively. Define $\hat{B}= \hat{B}^1-\hat{B}^2$; then $\hat{B}$ is $\F$-predictable, right-continuous and $P$-a.s.\ of finite variation, and of course indistinguishable from $B$.

  As a consequence, $\hat{M} := X-\hat{B}$ is right-continuous, $\F$-adapted and indistinguishable from $M$; in particular, it is still an $\mathbb{F}_+^P$-local martingale. By~\cite[Theorem~3]{Dellacherie.78}, there exists an $\mathbb{F}_+^P$-predictable localizing sequence $(\widetilde{\tau}_n)$ for $\hat{M}$. For any $\widetilde{\tau}_n$, there exists an $\mathbb{F}$-predictable stopping time $\tau_n$ such that $\widetilde{\tau}_n=\tau_n$ $P$-a.s.; cf.\ \cite[Theorem~IV.78, p.\,133]{DellacherieMeyer.78}. Thus, the sequence $(\tau_n)$ is a localizing sequence of $\mathbb{F}$-stopping times for the $\mathbb{F}_+^P$-local martingale $\hat{M}$. Since $\hat{M}$ is $\F$-adapted, we deduce from the tower property of the conditional expectation that $\hat{M}$ is an $\F$-local martingale. As a result, $X=\hat{M} + \hat{B}$ is a decomposition as required and we have shown that~(iii) implies~(i).

  The equivalence between (ii) and (iii) now follows because we can apply the equivalence of (i) and (iii) to the filtration $\F':=\F_+$.

  It remains to show the indistinguishability of the characteristics. Let  $(B,C,\nu)$ be $\F$-characteristics of $X$  and let $(B',C',\nu')$ be $\F_{+}^P$-characteristics.
   The second characteristic is the continuous part of the quadratic variation $[X]$, which can be constructed pathwise $P$-a.s.\ (see the proof of Proposition~\ref{prop:quadvarunabP}) and thus is independent of the filtration. As a result, $C=C'$ $P$-a.s.
  To identify the first characteristic, consider the process
  \begin{equation*}
    \widetilde{X}_t:=X_t-\sum_{0\leq s \leq t} \big(\Delta X_s- h(\Delta X_s) \big).
  \end{equation*}
  As $\widetilde{X}$ has uniformly bounded jumps, it is an $\F$-special semimartingale. Let $\widetilde{X}= X_0 + \widetilde{M} + B$ be the canonical decomposition with respect to $\F$ (cf.\ \cite[Theorem~7.2.6, p.\,160]{vonWeizsackerWinkler.90}). By the arguments in the first part of the proof, this is also the canonical decomposition with respect to $\F_{+}^P$ and thus $B=B'$ $P$-a.s.\ by the definition of the first characteristic.

  Next, we show that $\nu=\nu'$ $P$-a.s. To this end, we may assume that $\nu$ is already the $\F$-predictable compensator of $\mu^X$. (The existence of the latter follows from \cite[Theorem~II.1.8, p.\,66]{JacodShiryaev.03} and~\cite[Lemma~7, p.\,399]{DellacherieMeyer.82}.) Let us check that $\nu$ is also a predictable random measure with respect to $\F_+^P$. Let $W^P=W^P(\omega,t,x)$ be a $\cP^P\otimes\B(\R^d)$-measurable function; we claim that $W^P$ is indistinguishable from a $\cP\otimes\B(\R^d)$-measurable function $W$, in the sense that the set $\{ \omega \in \Omega \, | \,  W (\omega,t,x) \neq W^P (\omega,t,x) \mbox{ for some } (t,x)\}$ is $P$-null. To see this, consider first the case where $W^P(\omega,t,x)=H^P(\omega,t) J(x)$ with $H^P$ being $\P^P$-measurable and $J$ being $\B(\R^d)$-measurable. By \cite[Lemma~7, p.\,399]{DellacherieMeyer.82}, there exists a $\cP$-measurable process $H$ indistinguishable from $H^P$ and thus $W(\omega,t,x)= H(\omega,t) J(x)$ has the desired properties. The general case follows by a monotone class argument.
  Since $\nu$ is a predictable random measure with respect to $\F$, the process defined by
  \begin{equation*}
    (W\ast \nu)_t:= \int_0^t \int_{\R^d} W(s,x)\,\nu(ds,dx)
  \end{equation*}
  is $\cP$-measurable. As a result, the indistinguishable process $(W^P\ast \nu)$ is $\P^P$-measurable, showing that $\nu$ is a predictable random measure with respect to~$\F_+^P$.

  To see that $\nu$ is the compensator of the jump measure $\mu^X$ of $X$ with respect to $\F_+^P$, suppose that $W^P$ is nonnegative. Then by the indistinguishability of $W$ and $W^P$ and \cite[Theorem~II.1.8, p.\,66]{JacodShiryaev.03},
  \begin{equation*}
  E^P[(W^P\ast \nu)_\infty]= E^P[(W\ast \nu)_\infty]=  E^P[(W\ast \mu^X)_\infty]= E^P[(W^P\ast \mu^X)_\infty].
  \end{equation*}
  Now the uniqueness of the $\F_+^P$-compensator as stated in the cited theorem shows that $\nu=\nu'$ $P$-a.s. This completes the proof that $(B,C,\nu)=(B',C',\nu')$ $P$-a.s.

  Again, the argument for $\F_+$ is contained in the above as a special case, and so the proof of Proposition~\ref{prop:SemFunab} is complete.
\end{proof}

To study the measurability of $\fP_{sem}$, we need to express the semimartingale property in a way which is more accessible than the mere existence of a semimartingale decomposition. To this end, it will be convenient to use some facts which were developed in~\cite{BeiglbockSchachermayerVeliyev.11} to give an alternative proof of the Bichteler--Dellacherie theorem.

We continue to consider a c\`adl\`ag, $\R^d$-valued, $\F$-adapted process $X$ on an arbitrary filtered space $(\Omega,\cF,\F)$, but fix a finite time horizon $T>0$. Let $(\widetilde{X}_t)_{t\in [0,T]}$ be the process defined by
\begin{equation*}
  \widetilde{X}_t:= X_t-X_0-\sum_{0 \leq s \leq t} \Delta X_s \, \1_{\{|\Delta X_s|> 1\}}
\end{equation*}
and consider the sequence of $\F$-stopping times
\begin{equation*}
  T_m:= \inf\big\{ t\geq 0 \, \big| \, |\widetilde{X}_t|\geq m \ \mbox{or} \ |\widetilde{X}_{t-}|\geq m \big\}.
\end{equation*}
Moreover, for any $m \in \N$, define the process $(\widetilde{X}^m_t)_{t\in [0,T]}$ by
\[
  \widetilde{X}^m_t:= (m+1)^{-1} \widetilde{X}_{T_m\wedge t}.
\]
Given $P \in \fP(\Omega)$, we can consider the Doob decomposition of $\widetilde{X}^m$ sampled on the $n$-th dyadic partition of $[0,T]$ under $P$ and $\F_{+}$; namely,  $A^{m,P,n}:=0$ and
\begin{align*}
  A^{m,P,n}_{kT/2^n}&:=\sum_{j=1}^k E^P\big[\widetilde{X}^m_{jT/2^n}-\widetilde{X}^m_{(j-1)T/2^n}\, \big| \, \cF_{(j-1)T/2^n +} \big], \ \ \ 1\leq k\leq 2^n,\\
  M^{m,P,n}_{kT/2^n}&:= \widetilde{X}^m_{kT/2^n}-A^{m,P,n}_{kT/2^n},  \ \ \ 0\leq k \leq 2^n.
\end{align*}
Furthermore, given $c>0$, we define the $\F_+$-stopping times
\begin{align*}
\sigma_{m,n}(c)&:= \inf\bigg\{\frac{kT}{2^n} \, \bigg | \, \sum_{j=1}^k \Big| \widetilde{X}^m_{\frac{jT}{2^n}}-\widetilde{X}^m_{\frac{(j-1)T}{2^n}}\Big|^2 \geq c-4 \bigg\}, \\
\tau_{m,P,n}(c)&:= \inf\bigg\{\frac{kT}{2^n} \, \bigg | \, \sum_{j=1}^k \Big| A^{m,P,n}_{\frac{jT}{2^n}}-A^{m,P,n}_{\frac{(j-1)T}{2^n}}\Big| \geq c-2 \bigg\}.
\end{align*}

\begin{proposition}\label{prop:B/S/Vgen}
  Let $P \in \fP(\Omega)$. The process $(X_t)_{t\in [0,T]}$ is a $P$-$\F$-semi\-martingale if and only if for all $m \in\N$ and $\varepsilon>0$ there exists a constant $c=c(m,\varepsilon)>0$ such that
  \begin{equation}\label{eq:BSVcondition}
    P\big[\sigma_{m,n}(c)<\infty\big]< \frac{\varepsilon}{2} \quad\mbox{and} \quad P\big[\tau_{m,P,n}(c)<\infty\big]< \frac{\varepsilon}{2}\quad \mbox{for all }n\geq1.
  \end{equation}
\end{proposition}

\begin{proof}
  Clearly $X$ is a $P$-$\mathbb{F}$-semimartingale if and only if $\widetilde{X}^m$ has this property for all $m$. Moreover, by Proposition~\ref{prop:SemFunab}, this is equivalent to $\widetilde{X}^m$ being a $P$-$\F^P_+$-semimartingale.

  If $\widetilde{X}^m$ is a $P$-$\F_{+}^P$-semimartingale, \cite[Theorem~1.6]{BeiglbockSchachermayerVeliyev.11} implies that it satisfies the property ``no free lunch with vanishing risk and little investment'' introduced in~\cite[Definition~1.5]{BeiglbockSchachermayerVeliyev.11}. As $\sup_{t \in [0,T]} |\widetilde{X}^m_t|\leq 1$,  we deduce from \cite[Proposition~3.1]{BeiglbockSchachermayerVeliyev.11} that for any $\varepsilon>0$ there exists a constant $c=c(m,\varepsilon)>0$ such that~\eqref{eq:BSVcondition} holds.
  Conversely, suppose that there exist such constants; then, as $\sup_{t \in [0,T]} |\widetilde{X}^m_t|\leq 1$, the proof of~\cite[Theorem~1.6]{BeiglbockSchachermayerVeliyev.11} shows that $\widetilde{X}^m$ is a  $P$-$\F_{+}^P$-semimartingale.
\end{proof}

\begin{corollary}\label{co:P_sem}
  Under the conditions of Theorem~\ref{th:charactMbl}, the set
  \[
    \fP_{sem,T}=\{P\in \fP(\Omega)\,|\, \mbox{$(X_t)_{0\leq t\leq T}$ is a semimartingale on }(\Omega,\cF,\F,P)\}
  \]
  is Borel-measurable for every $T>0$, and so is $\fP_{sem}=\cap_{T\in\N} \fP_{sem,T}$.
\end{corollary}

\begin{proof}
  Let $T>0$; then Proposition~\ref{prop:B/S/Vgen} allows us to write $\fP_{sem,T}$ as
  \begin{align*}
  \bigcap_{m,k \in \N} \bigcup_{c \in \N} \bigcap_{n\in \N} \Big\{ P \in \fP(\Omega) \, \Big| \, P\big[\sigma_{m,n}(c)<\infty\big]
     +P\big[\tau_{m,P,n}(c)<\infty\big]< 1/k \Big\};
  \end{align*}
  hence, it suffices to argue that the right-hand side is measurable. Indeed, $\omega\mapsto \sigma_{m,n}(c)(\omega)$ and $(P,\omega) \mapsto \tau_{m,P,n}(c)(\omega)$ are measurable by Lemma~\ref{le:BorelcondexpPP}, so the required measurability follows by another application of the same lemma.
\end{proof}

\section{Measurable Doob-Meyer and Canonical Decompositions}\label{se:DoobMeyer}

In this section, we first obtain a version of the Doob--Meyer decomposition which is measurable with respect to the probability measure $P$. Then, we apply this result to construct the canonical decomposition of a bounded semimartingale with the same measurability; together with a localization argument, this will provide the first semimartingale characteristic $B^P$ in the subsequent section.
The conditions of Theorem~\ref{th:charactMbl} are in force; moreover, we fix a measurable set $\fP\subseteq \fP(\Omega)$. As the results of this section can be applied componentwise, we consider scalar processes without compromising the generality.

There are various proofs of the Doob--Meyer theorem, all based on compactness arguments, which use a passage to the limit from the elementary Doob decomposition in discrete time. The latter is measurable with respect to $P$ by Lemma~\ref{le:BorelcondexpPP}. Thus, the main issue is to go through a compactness argument while retaining measurability. Our Proposition~\ref{pr:mblKomlos} is tailored to that purpose, and it combines naturally with the proof of the Doob--Meyer decomposition given in~\cite{BeiglbockSchachermayerVeliyev.12}.

\begin{proposition}[Doob--Meyer]\label{prop:Doobmeyer}
  Let $(P,\omega,t)\mapsto S^P_t(\omega)$ be a measurable function such that for all $P\in\fP$, $S^P$ is a right-continuous, $\F_+$-adapted $P$-$\F^P_+$-submartingale of class D. There exists a measurable function $(P,\omega,t)\mapsto A^P_t(\omega)$ such that for all $P\in\fP$,
  \[
    S^P-S^P_0-A^P\quad \mbox{is a $P$-$\F^P_+$-martingale},
  \]
   $A^P$ is right-continuous, $\F_{+}$-adapted, $\F^P_{+}$-predictable and $P$-a.s.\ increasing.
\end{proposition}

\begin{proof}
  It suffices to consider a finite time horizon $T>0$; moreover, we may assume that $S_0^P=0$.
  For each $P \in \fP$ and $n \in \N$, consider the Doob decomposition of the process $(S^P_{jT/2^n})_{j=0,...,2^n}$, defined by  $A^{P,n}_0=0$ and
  \begin{align*}
    A^{P,n}_{kT/2^n}&= \sum_{j=1}^k E^P\Big[S^P_{jT/2^n}-S^P_{(j-1)T/2^n}\, \Big| \, \mathcal{F}_{(j-1)T/2^n +} \Big], \ \ \ \ 1\leq k\leq 2^n,\\
    M^{P,n}_{kT/2^n}&= S^P_{kT/2^n}- A^{P,n}_{kT/2^n}, \ \ \ \ 0\leq k\leq 2^n.
  \end{align*}
  Note that $(A^{P,n}_{jT/2^n})_{j=0,...,2^n}$ has $P$-a.s.\ increasing paths and that $(P,\omega)\mapsto A^{P,n}_{kT/2^n}(\omega)$ is measurable by Lemma \ref{le:BorelcondexpPP}. As a consequence, $(P,\omega)\mapsto M^{P,n}_T(\omega)$ is measurable as well.
  We deduce from \cite[Lemma~2.2]{BeiglbockSchachermayerVeliyev.12} that for each $P \in \fP$ the sequence $(M^{P,n}_T)_{n \in \N}\subseteq L^1(P)$ is uniformly integrable. Therefore, we can apply Proposition~\ref{pr:mblKomlos} to obtain an $L^1(P)$-convergent sequence of convex combinations
  \begin{equation*}
  \mathcal{M}^{P,n}_T:= \sum_{i=n}^{N^P_n} \lambda^{P,n}_i\,M^{P,i}_T
  \end{equation*}
  which are measurable in $(P,\omega)$. By Lemma \ref{le:L1messbar}, we can find a version $\mathcal{M}^P_T$ of the limit which is again jointly measurable in $(P,\omega)$.

  On the strength of Lemma~\ref{le:BorelcondexpPP} and Lemma~\ref{le:cadlagversionm}, we can find a measurable function $(P,\omega,t)\mapsto M^P_t(\omega)$ such that for each $P \in \fP$,  $(M^P_t)_{t \in [0,T]}$ is a right-continuous $P$-$\F_{+}$-martingale and a $P$-modification of $(E^P[\mathcal{M}^{P}_T \,| \, \mathcal{F}_{t+}])_{0 \leq t \leq T}$.
  We define $A^P$ by
  \begin{equation*}
  A^P_t:= S^P_t -M^P_t;
  \end{equation*}
  then $A^P$ is right-continuous and $\F_{+}$-adapted and $(P,\omega,t)\mapsto A^P_t(\omega)$ is measurable. Following the arguments in~\cite[Section~2.3]{BeiglbockSchachermayerVeliyev.12}, we see that $A^P$ is $P$-a.s.\ increasing and $P$-indistinguishable from a $\P$-measurable process, hence predictable with respect to $\F^P_{+}$.
\end{proof}

We can now construct the compensator of a process with integrable variation. We recall the filtration $\widehat{\F}$ on $\fP\times \Omega$ introduced in~\eqref{eq:defHatOmega}.

\begin{corollary}[Compensator]\label{co:Doobmeyer}
  Let $(P,\omega,t)\mapsto S^P_t(\omega)$ be a right-continuous $\widehat{\F}_+$-adapted process such that for all $P\in\fP$, $S^P$ is an $\F_+$-adapted process of $P$-integrable variation. There exists a measurable function $(P,\omega,t)\mapsto A^P_t(\omega)$ such that for all $P\in\fP$,
  \[
    S^P-S^P_0-A^P\quad \mbox{is a $P$-$\F^P_+$-martingale},
  \]
   $A^P$ is right-continuous, $\F_{+}$-adapted, $\F^P_{+}$-predictable and $P$-a.s.\ of finite variation.
\end{corollary}

\begin{proof}
  We may assume that $S^P_0=0$. By Lemma~\ref{le:productm}, $(P,\omega,t)\mapsto S^P_t(\omega)$ is measurable. Thus, if $S^P$ is $P$-a.s.\ increasing for all $P\in\fP$, Proposition~\ref{prop:Doobmeyer} immediately yields the result. Therefore, it suffices to show that there exists a decomposition
  \[
    S^P=S^{1,P}-S^{2,P}\quad P\as
  \]
  into $\F_{+}$-adapted, $P$-integrable processes having right-continuous and $P$-a.s.\ increasing paths such that $(P,\omega,t)\mapsto S^{i,P}_t(\omega)$ is measurable.
  Let $\Var(S^P)$ denote the total variation process of $S^P$. By the right-continuity of $S^P$, we have
  \begin{equation*}
  \Var(S^P)_t(\omega)=  \lim_{n \to \infty} \sum_{k=1}^{2^n}  \big|S^P_{kt/2^n}(\omega)-S^P_{(k-1)t/2^n}(\omega) \big|\quad \mbox{for all $(P,\omega,t)$.}
  \end{equation*}
   In particular, $\Var(S^P)$ is $\F_{+}$-adapted and $(P,\omega,t)\mapsto \Var(S^P)_t(\omega)$ is $\widehat{\F}_+$-adapted. For each $P \in \fP$, we define
  \begin{equation*}
    \sigma^P:= \inf \big\{t \geq 0 \, \big| \, \Var(S^P)_t= \infty \big\}.
  \end{equation*}
  The identity
  \[
    \{\sigma^P<t\}= \bigcup_{q\in\Q,\, q<t} \{\Var(S^P)_q = \infty\}
  \]
  shows that $(P,\omega)\mapsto \sigma^P(\omega)$ is an $\widehat{\F}_+$-stopping time and in particular measurable.
  As $S^P$ is of $P$-integrable variation, we have $\sigma^P=\infty$ $P$-a.s. Using Lemma~\ref{le:productm} and the fact that $\Var(S^P)\1_{[\![0,\sigma^P[\![}$ is right-continuous, it follows that the processes
   \begin{equation*}
   S^{1,P}:= \frac{\Var(S^P)+S^P}{2}\,\1_{[\![0,\sigma^P[\![}, \ \ \ \  S^{2,P}:= \frac{\Var(S^P)-S^P}{2}\,\1_{[\![0,\sigma^P[\![}
  \end{equation*}
  have the required properties.
\end{proof}

In the second part of this section, we construct the canonical decomposition of a bounded semimartingale. Ultimately, this decomposition can be obtained from the discrete Doob decomposition, a compactness argument and the existence of the compensator for bounded variation processes. Hence, we will combine Proposition~\ref{pr:mblKomlos} and the preceding Corollary~\ref{co:Doobmeyer}.
The following lemma is an adaptation of the method developed in~\cite{BeiglbockSchachermayerVeliyev.11} to our needs; it contains the mentioned compactness argument. We fix a finite time horizon $T>0$.

\begin{lemma}\label{le:thm1.6B/S/V}
Let $S=(S_t)_{t\in[0,T]}$ be a c\`adl\`ag, $\F_+$-adapted process with $S_0=0$ and $\sup_{t \in [0,T]} |S_t| \leq 1$ such that $S$ is a $P$-$\F^P_+$-semimartingale for all $P\in\fP$. For all $\varepsilon>0$ and $P \in \fP$ there exist
  \begin{enumerate}[topsep=3pt, partopsep=0pt, itemsep=2pt,parsep=2pt]
    \item a $[0,T]\cup\{\infty\}$-valued $\F_{+}$-stopping time $\alpha^P$ such that $(P,\omega)\mapsto \alpha^P(\omega)$ is an $\widehat{\F}_+$-stopping time and
    \[
      P[\alpha^P<\infty] \leq\varepsilon,
    \]
    \item a constant $c^P$ and right-continuous, $\F_{+}$-adapted processes $\mathcal{A}^{P}$, $\mathcal{M}^{P}$ with $\mathcal{A}^P_0=\mathcal{M}^P_0=0$ such that $(P,\omega,t)\mapsto (\mathcal{A}^{P}_t(\omega),\mathcal{M}^{P}_t(\omega))$ is $\widehat{\F}_+$-adapted,
        \[
          \mathcal{M}^{P}\;\;\mbox{is a $P$-$\F_{+}$-martingale}\quad\mbox{and}\quad\Var(\mathcal{A}^{P})\leq c^P\;\; P\as
        \]
  \end{enumerate}
  such that
  \[
    \mathcal{M}^{P}_t + \mathcal{A}^{P}_t =S_{{\alpha^P}\wedge t},\quad t\in [0,T].
  \]
\end{lemma}

\begin{proof}
  This lemma is basically a version of~\cite[Theorem~1.6]{BeiglbockSchachermayerVeliyev.11} with added measurability in $P$; we  only give a sketch of the proof. The first step is to obtain a version of~\cite[Proposition~3.1]{BeiglbockSchachermayerVeliyev.11}: For $P\in\fP$ and $n\in\N$, consider the Doob decomposition of the discrete-time process $(S_{jT/2^n})_{j=0,...,2^n}$ with respect to $P$ and $\F_+$, defined by  $A^{P,n}_0=0$ and
  \begin{align*}
   A^{P,n}_{kT/2^n}&:= \sum_{j=1}^k E^P\Big[S_{jT/2^n}-S_{(j-1)T/2^n}\, \Big| \, \mathcal{F}_{(j-1)T/2^n +} \Big], \ \ \ \ 1 \leq k \leq 2^n,\\
   M^{P,n}_{kT/2^n}&:= S_{kT/2^n}- A^{P,n}_{kT/2^n}, \ \ \ \ 0\leq k \leq 2^n.
  \end{align*}
  By adapting the proof of~\cite[Proposition~3.1]{BeiglbockSchachermayerVeliyev.11} and using Lemma~\ref{le:BorelcondexpPP}, one shows that for all $\eps>0$ and $P\in\fP$ there exist a constant $c^P \in \N$ and a sequence of $\big\{T/2^n,...,(2^n-1)T/2^n,T\big\}\cup \{\infty\}$-valued $\F_{+}$-stopping times $(\rho_{P,n})_{n \in \N}$ such that $(P,\omega)\mapsto \rho_{P,n}(\omega)$ is an $\widehat{\F}_+$-stopping time,
  \[
    P[\rho_{P,n}<\infty]<\varepsilon
  \]
  and
  \[
  \sum_{j=1}^{\frac{2^n}{T}(\rho_{P,n}\wedge T)} \big| A^{P,n}_{jT/2^n}-A^{P,n}_{(j-1)T/2^n} \big| \leq c^P, \quad 
  \big\Vert M^{P,n}_{T \wedge \rho_{P,n}} \big\Vert^2_{L^2(P)} \leq c^P 
  \]
  for all $n\in\N$. The second step is to establish the following assertion: for all $\eps>0$ and $P\in\fP$ there exist a constant $c^P \in \N$, a $[0,T]\cup\{\infty\}$-valued $\F_{+}$-stopping time $\alpha^P$ such that $(P,\omega)\mapsto \alpha^P(\omega)$ is an $\widehat{\F}_+$-stopping time, and a sequence of right-continuous, $\F_{+}$-adapted  processes $(\mathcal{A}^{P,k})_{k \in \N}$ and $(\mathcal{M}^{P,k})_{k \in \N}$ on $[0,T]$ which are measurable in $(P,\omega,t)$, such that $(\mathcal{M}^{P,k}_t)_{0 \leq t \leq T}$ is a $P$-$\F_{+}$-martingale and
  \begin{align*}
  \big(\mathcal{M}^{P,k}\big)_t^{\alpha^P} + \big(\mathcal{A}^{P,k}\big)_t^{\alpha^P} = S_t^{\alpha^P},\quad\quad
  &P\big[\alpha^P<\infty\big] \leq\varepsilon,  \\
   \sum_{j=1}^{2^k} \Big|\big(\mathcal{A}^{P,k}\big)_{jT/2^k}^{\alpha^P} - \big(\mathcal{A}^{P,k}\big)_{(j-1)T/2^k}^{\alpha^P} \Big| \leq c^P  \ P\mbox{-a.s.},\quad\quad &\Big\Vert \big(\mathcal{M}^{P,k}\big)_{t}^{\alpha^P} \Big\Vert^2_{L^2(P)} \leq c^P.
  \end{align*}
  Here the first equality holds for all $\omega$ rather than $P$-a.s.\ and the usual notation for the ``stopped process'' is used; for instance, $S_t^{\alpha^P}=S_{\alpha^P\wedge t}$. To derive this assertion from the first step, we combine the arguments in the proof of~\cite[Proposition~3.6]{BeiglbockSchachermayerVeliyev.11} with Lemma~\ref{le:BorelcondexpPP}, Lemma~\ref{le:L1messbar}, Proposition~\ref{pr:mblKomlos} and Lemma~\ref{le:cadlagversionm}.

  Finally, to derive Lemma~\ref{le:thm1.6B/S/V} from the preceding step, we adapt the proof of~\cite[Theorem~1.6]{BeiglbockSchachermayerVeliyev.11}, again making crucial use of Lemma~\ref{le:L1messbar} and Proposition~\ref{pr:mblKomlos}.
\end{proof}

\begin{proposition}[Canonical Decomposition]\label{prop:predFVbdd}
  Let $S$ be a c\`adl\`ag, $\F_+$-adapted process with $S_0=0$ and $\sup_{t\geq0} |S_t| \leq 1$ such that $S$ is a $P$-$\F^P_+$-semimartingale for all $P\in\fP$.
  There exists a measurable function $(P,\omega,t)\mapsto B^P_t(\omega)$ such that for all $P\in\fP$, 
  \[
    S-B^P\quad \mbox{is a $P$-$\F^P_{+}$-martingale},
  \]
  $B^P$ is right-continuous, $\F_{+}$-adapted, $\F^P_{+}$-predictable and $P$-a.s.\ of finite variation.
\end{proposition}

\begin{proof}
  We first fix $T>0$ and consider the stopped process $Y= S^T$. For each $n\in\N$, let $\alpha^{P,n}$, $\mathcal{M}^{P,n}$ and $\mathcal{A}^{P,n}$ be the stopping times and processes provided by Lemma~\ref{le:thm1.6B/S/V} for the choice $\eps=2^{-n}$; that is, $P[\alpha^{P,n}< \infty]<2^{-n}$ and
  \begin{equation*}
  Y^{\alpha^{P,n}} = \mathcal{M}^{P,n} + \mathcal{A}^{P,n}.
  \end{equation*}
  By Corollary~\ref{co:Doobmeyer}, we can construct the compensator of  $\mathcal{A}^{P,n}$ with respect to $P$-$\F^P_{+}$, denoted by $\{\mathcal{A}^{P,n}\}^P$, such that $\{\mathcal{A}^{P,n}\}^P$ is right-continuous, $\F_{+}$-adapted and $P$-a.s.\ of finite variation, and $(P,\omega,t)\mapsto \{\mathcal{A}^{P,n}\}^P_t(\omega)$ is measurable. We define the process $\overline{\mathcal{M}}^{P,n}$ by
  \begin{equation*}
  \overline{\mathcal{M}}^{P,n} = \mathcal{M}^{P,n} + \mathcal{A}^{P,n}- \{\mathcal{A}^{P,n}\}^P.
  \end{equation*}
  By construction, $\overline{\mathcal{M}}^{P,n}$ is a right-continuous, $\F_{+}$-adapted $P$-$\F^P_{+}$-martingale and $(P,\omega,t)\mapsto \overline{\mathcal{M}}^{P,n}_t(\omega)$ is measurable. Furthermore,
  \[
    Y^{\alpha^{P,n}} = \overline{\mathcal{M}}^{P,n} + \{\mathcal{A}^{P,n}\}^P
  \]
  is the canonical decomposition of the $P$-$\F^P_{+}$-semimartingale $Y^{\alpha^{P,n}}$.

  We have $\sum_{n \in \N} P\{\alpha^{P,n}< \infty\}<\infty$ for each $P \in \fP$. By the Borel--Cantelli Lemma, this implies that $\lim_{n \to \infty} \alpha^{P,n} = \infty$ $P$-a.s.
  Let
  \begin{equation*}
    \beta^{P,n}:= \inf_{k\geq n} \alpha^{P,k}.
  \end{equation*}
  Then $\beta^{P,n}$ are $\F_{+}$-stopping times increasing to infinity $P$-a.s.\ for each $P$ and $(P,\omega)\mapsto \beta^{P,n}(\omega)$ is an $\widehat{\F}_+$-stopping time for each $n$. As $\beta^{P,n+1}\wedge \alpha^{P,n}=\beta^{P,n}$, we have
  \begin{equation*}
    Y^{\beta^{P,n}} = \big(Y^{\alpha^{P,n}}\big)^{\beta^{P,n+1}} = \big(\overline{\mathcal{M}}^{P,n}\big)^{\beta^{P,n+1}} + \big(\{\mathcal{A}^{P,n}\}^P\big)^{\beta^{P,n+1}},
  \end{equation*}
  which is the canonical decomposition of $Y^{\beta^{P,n}}$. Thus, by uniqueness of the canonical decomposition,
  \begin{align*}
  Y&
  &= \sum_{n=1}^\infty \big(\overline{\mathcal{M}}^{P,n}\big)^{\beta^{P,n+1}} \,\1_{[\![\beta^{P,n-1},\beta^{P,n}[\![} +  \sum_{n=1}^\infty \big(\{\mathcal{A}^{P,n}\}^P\big)^{\beta^{P,n+1}}  \,\1_{[\![\beta^{P,n-1},\beta^{P,n}[\![}
  \end{align*}
  is the canonical decomposition of $Y$, where we have set $\beta^{P,0}:=0$. Denote the two sums on the right-hand side by $M^{P,T}$ and $B^{P,T}$, respectively, and recall that $Y=S^T$. The decomposition of the full process $S$ is then given by
  \[
    S=M^P+B^P:=\sum_{T=1}^\infty M^{P,T} \,\1_{[\![T-1,T[\![}+\sum_{T=1}^\infty B^{P,T} \,\1_{[\![T-1,T[\![}.
  \]
  By construction, these processes have the required properties.
\end{proof}

\section{Measurable Semimartingale Characteristics}\label{se:characteristics}

In this section, we construct a measurable version of the characteristics $(B^P,C,\nu^P)$ of $X$ as stated in Theorem~\ref{th:charactMbl}. The conditions of that theorem are in force throughout; in particular, $X$ is a c\`adl\`ag, $\F$-adapted process. We recall that the set $\fP_{sem}$ of all $P\in\fP(\Omega)$ under which $X$ is a semimartingale is measurable (Corollary~\ref{co:P_sem}) and that a truncation function $h$ has been fixed. When we refer to the results of Section~\ref{se:auxResults}, they are to be understood with the choice $\fP=\fP_{sem}$.

As mentioned in the preceding section, the existence of the first characteristic $B^P$ is a consequence of Proposition~\ref{prop:predFVbdd}.

\begin{corollary}\label{co:predFVm}
  There exists a measurable function $\fP_{sem}\times \Omega\times \R_+ \to \R^d$, $(P,\omega,t)\mapsto B^P_t(\omega)$
  such that for all $P\in\fP_{sem}$, $B^P$ is an $\F_{+}$-adapted, $\F^P_{+}$-predictable process with right-continuous, $P$-a.s.\ finite variation paths, and $B^P$ is a version of the first characteristic of $X$ with respect to $P$.
\end{corollary}

\begin{proof}
  We may assume that $X_0=0$. Let
  \begin{equation*}
    \widetilde{X}_t:= X_t-\sum_{0\leq s \leq t} \big(\Delta X_s- h(\Delta X_s)\big),
  \end{equation*}
  $T_0=0$ and $T_m= \inf \{ t \geq 0 \, | \, |\widetilde{X}_t| >m \}$. As $\widetilde{X}$ has c\`adl\`ag paths, each $T_m$ is an $\F_{+}$-stopping time and $T_m\to\infty$. Define
  \begin{equation*}
    \widetilde{X}^m = \widetilde{X}_{\cdot\wedge T_m};
  \end{equation*}
  then $\widetilde{X}^m$ is a c\`adl\`ag, $\F_{+}$-adapted $P$-$\F^P_{+}$ semimartingale for each $P \in \fP_{sem}$ and $|\widetilde{X}^m|\leq m+\|h\|_{\infty}$. We use Proposition~\ref{prop:predFVbdd} to obtain the corresponding predictable finite variation process $B^{m,P}$ of the canonical decomposition of $\widetilde{X}^m$, and then
  \begin{equation*}
    B^P = \sum_{m\geq1} B^{m,P}  \,\1_{[\![T_{m-1},T_m[\![}
  \end{equation*}
  has the desired properties.
\end{proof}

The next goal is to construct the third characteristic of $X$, the compensator $\nu^P$ of the jump measure of $X$, and its decomposition as stated in Theorem~\ref{th:charactMbl}. (The second characteristic is somewhat less related to the preceding results and thus treated later on.) To this end, we first provide a variant  of the disintegration theorem for measures on product spaces. As it will be used for the decomposition of $\nu^P$, we require a version where the objects depend measurably on an additional parameter (the measure $P$). We call a kernel stochastic if its values are probability measures, whereas finite kernel refers to the values being finite measures. A Borel space is (isomorphic to) a Borel subset of a Polish space.

\begin{lemma}\label{le:decompkern}
  Let $(G,\cG)$ be a measurable space, $(Y,\mathcal{Y})$ a separable measurable space and $(Z,\B(Z))$ a Borel space. Moreover, let $\kappa\big(g,d(y,z)\big)$ be a finite kernel on $(Y \times Z, \mathcal{Y}\otimes \B(Z))$ given $(G,\cG)$ and let $\widehat{\kappa}(g,dy)$ be its marginal on~$Y$,
  \[
    \widehat{\kappa}(g,A):= \kappa(g,A\times Z),\quad A \in \mathcal{Y}.
  \]
  There exists a stochastic kernel $\alpha\big((g,y),dz\big)$ on $(Z,\B(Z))$ given $(G\times Y,\cG\otimes\mathcal{Y})$ such that
  \[
  \kappa(g,A\times B) = \int_{A} \alpha\big((g,y),B\big)\,\widehat{\kappa}(g,dy), \ \ \ A \in \mathcal{Y}, \ B \in \B(Z),\ g\in G.
  \]
\end{lemma}

\begin{proof}
  This result can be found e.g.\ in \cite[Proposition~7.27, p.\,135]{BertsekasShreve.78}, in the special case where $Y$ is a Borel space (and $\kappa\big(g,d(y,z)\big)$ is a stochastic kernel). In that case, one can identify $Y$ with an interval and the proof of \cite[Proposition~7.27, p.\,135]{BertsekasShreve.78} makes use of dyadic partitions generating $\mathcal{Y}$. In the present case, we can give a similar proof where we use directly the separability of $\mathcal{Y}$; namely, we can find a refining sequence of finite partitions of $Y$ which generates $\mathcal{Y}$ and apply martingale convergence arguments to the corresponding sequence of finite $\sigma$-fields. The details are omitted.
\end{proof}

In order to apply the disintegration result with $(Y,\mathcal{Y})=(\Omega\times\R_+,\cP)$, we need the following observation.

\begin{lemma}\label{le:predictableSeparable}
  The predictable $\sigma$-field $\cP$ is separable.
\end{lemma}

\begin{proof}
  The $\sigma$-field $\cP$ is generated by the sets
  \[
    \{0\}\times A, \ A\in \cF_{0-}\quad\mbox{ and }\quad (s,t]\times A, \ \ A\in \cF_{s-}, \  0<s<t \in \Q;
  \]
  cf.\ \cite[Theorem~IV.67, p.\,125]{DellacherieMeyer.78}. Since $\cF_{0-}$ is trivial and $\cF_s$ is separable for $s\geq0$, it follows that each $\cF_{s-}$ is separable as well. Let $(A_s^n)_{n\geq1}$ be a generator for $\cF_{s-}$; then
  \[
    \{0\}\times A_0^n \quad\mbox{ and }\quad (s,t]\times A^n_s, \ \   0<s<t \in \Q,\ n\geq 1
  \]
  yield a countable generator for $\cP$.
\end{proof}

We can now construct the third characteristic and its decomposition. For the following statement, recall the set $\overline{\cL}$ from~\eqref{eq:defOverlineL} and that it has been endowed with its Borel $\sigma$-field.

\begin{proposition}\label{prop:mversioncomp}
  There exists a measurable function
  \begin{equation*}
    \fP_{sem}\times \Omega \to \overline{\cL}, \ \ \ (P, \omega) \mapsto \nu^P(\omega,dt,dx)
  \end{equation*}
  such that for all $P \in \fP_{sem}$, the $\F_+^P$-predictable random measure $\nu^P(\cdot,dt,dx)$ is the $P$-$\F_+^P$-compensator of $\mu^X$.
  Moreover, there exists a decomposition
  \begin{equation*}
    \nu^P(\cdot,dt,dx) = K^P(\cdot,t,dx)\,dA^P_t \quad P\as
  \end{equation*}
  where
  \begin{enumerate}[topsep=3pt, partopsep=0pt, itemsep=2pt,parsep=2pt]
    \item $(P,\omega,t)\mapsto A^P_t(\omega)$ is measurable and for all $P \in \fP_{sem}$, $A^P$ is an $\F_{+}$-adapted, $\F^P_+$-predictable, $P$-integrable process with right-continuous and $P$-a.s.\ increasing paths,
    \item $(P,\omega,t)\mapsto K^P(\omega,t,dx)$ is a kernel on $(\R^d,\B(\R^d))$ given $(\fP_{sem}\times \Omega\times\R_+, \B(\fP_{sem})\otimes \cF\otimes\cB(\R_+))$ and for all $P \in \fP_{sem}$, $(\omega,t)\mapsto K^P(\omega,t,dx)$ is a kernel on $(\R^d,\B(\R^d))$ given $(\Omega\times\R_+, \P^P)$.
  \end{enumerate}
\end{proposition}

\begin{proof}
  We use the preceding results to adapt the usual construction of the compensator, with $P\in\fP_{sem}$ as an additional parameter.
  By a standard fact recalled in Lemma~\ref{le:mversioncomp} below, there is a $\mathcal{P}\otimes \mathcal{B}(\R^d)$-measurable function $V>0$ such that $0\leq V\ast \mu^X \leq 1$; recall the notation $V\ast \mu^X:=\int_0^\cdot \int_{\R^d} V(s,x)\,\mu^X(ds,dx)$. Define $A:= V \ast \mu^X$. We observe that $A$ is a c\`adl\`ag, $\F_+$-adapted process, uniformly bounded and increasing; thus, it is a $P$-$\F^P_+$-submartingale of class D for any $P \in \fP_{sem}$. By Proposition~\ref{prop:Doobmeyer}, we can construct the predictable process of the Doob-Meyer decomposition of $A$ with respect to $P$ and $\F^P_{+}$, denoted by $A^P$, such that $A^P$ is  $P$-integrable, $\F^P_+$-predictable, $\F_{+}$-adapted with right-continuous, $P$-a.s.\ increasing paths and $(P,\omega,t)\mapsto A^P_t(\omega)$ is measurable. Define a kernel on $(\Omega\times\R_+\times \R^d, \cP\otimes\B(\R^d))$ given $\big(\fP_{sem}, \B(\fP_{sem})\big)$ by
  \begin{equation*}
    m^P(C):= E^P \big[ \big(V\,\1_C \ast \mu^X\big)_\infty \big], \ \ \ C \in \cP\otimes\B(\R^d).
  \end{equation*}
  Note that each measure $m^P(\cdot)$ is a sub-probability. Consider the set
  \begin{align*}
    G&:=\big\{ (P,\omega) \in \fP_{sem} \times \Omega \, \big| \, t\mapsto A^P_t(\omega) \mbox{ is increasing}\big\}\\
       &\,= \bigcap_{s<t\in\Q} \big\{ (P,\omega) \in \fP_{sem} \times \Omega \, \big| \, A^P_s(\omega)<A^P_t(\omega)\big\};
  \end{align*}
  the second equality is due to the right-continuity of $A^P$ and shows that $G \in \B(\fP_{sem})\otimes\cF$. Moreover, the sections of $G$ satisfy
  \[
    P\{ \omega \in \Omega \, \big| \, (P,\omega) \in G\}=1,\quad P\in \fP_{sem}.
  \]
  Thus, the (everywhere increasing, but not $\F_+$-adapted) process
  \[
    \bar{A}^P_t(\omega):=A^P_t(\omega) \1_{G}(P,\omega)
  \]
  is $P$-indistinguishable from $A^P$ and in particular $\F^P_+$-predictable, while the map $(P,\omega,t)\mapsto \bar{A}^P_t(\omega)$ is again measurable.
  We define another finite kernel on  $(\Omega\times\R_+, \mathcal{P})$ given $\big(\fP_{sem}, \B(\fP_{sem})\big)$ by
  \begin{equation*}
  \widehat{m}^P(D)= E^P\bigg[\int_0^\infty \1_D(t,\omega) \, d\bar{A}^P_t(\omega)\bigg], \ \ \ D \in \mathcal{P}.
  \end{equation*}
  As in the proof of \cite[Theorem II.1.8, p.\,67]{JacodShiryaev.03}, we have $\widehat{m}^P(D) =m^P(D\times \R^d)$ for any $D \in \mathcal{P}$; that is, $\widehat{m}^P(d\omega,dt)$ is the marginal of $m^P(d\omega,dt,dx)$ on $(\Omega\times\R_+, \mathcal{P})$.

  Since  $(\Omega\times\R_+,\P)$ is separable by Lemma~\ref{le:predictableSeparable}, we may apply the disintegration result of Lemma~\ref{le:decompkern} to obtain a stochastic kernel $\alpha^P(\omega,t,dx)$ on  $(\R^d,\mathcal{B}(\R^d))$ given
  $(\fP_{sem} \times \Omega\times\R_+, \B(\fP_{sem}) \otimes \mathcal{P})$  such that
  \begin{equation*}
    m^P(d\omega,dt,dx)=\alpha^P(\omega,t,dx) \,\widehat{m}^P(d\omega,dt).
  \end{equation*}
  Define a kernel $\widetilde{K}^P(\omega,t,dx)$ on $(\R^d,\mathcal{B}(\R^d))$ given
  $(\fP_{sem} \times \Omega\times\R_+, \mathcal{B}(\fP_{sem}) \otimes \mathcal{P})$  by
  \begin{equation}\label{eq:defKtildeKernel}
    \widetilde{K}^P(\omega,t,E) :=\int_E V(\omega,t,x)^{-1}\,\alpha^P(\omega,t,dx), \ \ \ E \in \B(\R^d).
  \end{equation}
  Moreover, let $\widetilde{\nu}^P(\omega,dt,dx):=  \widetilde{K}^P(\omega,t,dx) \,d\bar{A}^P_t(\omega)$ and define the set
  \begin{multline*}
    G':= \bigg\{ (P,\omega) \in G \, \Big| \, \int_0^N \int_{\R^d} |x|^2 \wedge 1 \, \widetilde{\nu}^P(\omega,dt,dx)<\infty\,\forall\, N \in \N, \\
   \; \widetilde{\nu}^P(\omega,\R_+,\{0\})= 0 = \widetilde{\nu}^P(\omega,\{0\},\R^d) \bigg\}.
  \end{multline*}
  We observe that  $G' \in \B(\fP_{sem})\otimes \cF$. Moreover, by \cite[Theorem~II.1.8, p.\,66]{JacodShiryaev.03} and its proof,
  \begin{equation}\label{eq:sectionsNullG2}
    P\{ \omega \in \Omega \, | \, (P,\omega) \in G'\}=1,\quad P\in \fP_{sem}.
  \end{equation}
  Define the kernel $K^P(\omega,t,dx)$ on $(\R^d,\mathcal{B}(\R^d))$ given
  $(\fP_{sem} \times \Omega\times\R_+, \mathcal{B}(\fP_{sem}) \otimes \cF\otimes\cB(\R_+))$  by
  \begin{equation*}
    K^P(\omega,t,E) :=\widetilde{K}^P(\omega,t,E) \, \1_{G'}(P,\omega), \ \ \ \ E \in \B(\R^d).
  \end{equation*}
  We see from~\eqref{eq:sectionsNullG2} that for fixed $P \in \fP_{sem}$, $K^P(\omega,t,dx)$ is also a kernel on $(\R^d,\mathcal{B}(\R^d))$ given $(\Omega\times\R_+,\P^P)$. Finally, we set
  \begin{equation*}
    \nu^P(\omega,dt,dx) := K^P(\omega,t,dx) \,d\bar{A}^P_t(\omega),
  \end{equation*}
  which clearly entails that $\nu^P(\cdot,dt,dx)= K^P(\cdot,t,dx) \,dA^P_t$ $P$-a.s.
  By construction, $\nu^P(\omega,dt,dx) \in \overline{\cL}$ for each $(P,\omega)\in \fP_{sem}\times \Omega$. Moreover, we deduce from \cite[Theorem~II.1.8, p.\,66]{JacodShiryaev.03} that $\nu^P(\omega,dt,dx)$ is the $P$-$\F^P_{+}$-compensator of $\mu^X$ for each $P \in \fP_{sem}$. It remains to show that $(P,\omega)\mapsto \nu^P(\omega,dt,dx)$ is measurable. By Lemma \ref{le:kernlevy}, it suffices to show that given a Borel function $f$ on $\R_+ \times \R^d$,  the map
  \[
    (P,\omega) \mapsto f(t,x)\ast \nu^P(\omega,dt,dx)
  \]
  is measurable. Suppose first that $f$ is of the form $f(t,x)=g(t)\,h(x)$, where $g$ and $h$ are measurable functions. Then
   \begin{align*}
  f(t,x)\ast \nu^P(\omega,dt,dx) &=  \int_0^\infty\int_{\R^d} f(t,x)\,K^P(\omega,t,dx)\,d\bar{A}^P_t(\omega)\\
   &=\int_0^\infty g(t) \, \int_{\R^d} h(x)\,K^P(\omega,t,dx)\,d\bar{A}^P_t(\omega)
  \end{align*}
  is measurable in $(P,\omega)$. The case of a general function $f$ follows by a monotone class argument, which completes the proof.
\end{proof}

The following standard fact was used in the preceding proof.

\begin{lemma}\label{le:mversioncomp}
  Let $S$ be a c\`adl\`ag, $\F_+$-adapted process. There exists a strictly positive $\mathcal{P}\otimes \mathcal{B}(\R^d)$-measurable function $V$ such that $0\leq V\ast \mu^S \leq 1$.
\end{lemma}

\begin{proof}
  Let $H_n:=\{ x \in \R^d \, | \, |x|> 2^{-n} \}$ for $n\in\N$; then $\cup_{n} H_n = \R^d\setminus\{0\}$. Define $T_{n,0}=0$ and
  \begin{equation*}
  T_{n,m}:= \inf \big \{t \geq T_{n,m-1} \, \big | \, |S_t - S_{T_{n,m-1}}|>2^{-(n+1)} \big\}.
  \end{equation*}
  As $S$ is c\`adl\`ag, each $T_{n,m}$ is an $\F_{+}$-stopping time. Set $G_{n,0}:= \Omega\times\R_+ \times \{0\}$ and
  \begin{equation*}
  G_{n,m}:= [\![0,T_{n,m}]\!] \times  H_n \in \mathcal{P}\otimes \mathcal{B}(\R^d);
  \end{equation*}
  recall that the predictable $\sigma$-field associated with $\F_{+}$ coincides with $\cP$. Then, 
  $\cup_{n,m} G_{n,m}= \Omega\times\R_+ \times \R^d$ and
  \begin{equation*}
  V(\omega,t,x):= \sum_{n \geq 1} 2^{-n} \bigg(\1_{G_{n,0}}(\omega,t,x) + \sum_{m \geq 1} \frac{2^{-m}}{m} \1_{G_{n,m}}(\omega,t,x) \bigg)
  \end{equation*}
  has the required properties.
\end{proof}

The final goal of this section is to establish an aggregated version of the second characteristic; that is, a single process $C$ rather than a family $(C^P)_{P\in\fP_{sem}}$. By its definition, $C$ is the quadratic variation of the continuous local martingale part of $X$ under each $P\in\fP_{sem}$; however, the martingale part depends heavily on $P$ and thus would not lead to an aggregated process $C$. Instead, we shall obtain $C$ as the continuous part of the (optional) quadratic variation $[X]$ which is essentially measure-independent. For future applications, we establish two versions of $C$: one is $\F$-predictable but its paths are irregular on an exceptional set; the other one, denoted $\bar{C}$, has regular paths and is predictable for the augmentation of $\F$ by the collection of $\cF$-measurable $\fP_{sem}$-polar sets. More precisely, we let
\[
  \cN_{sem}=\{A\in \cF\,|\, P(A)=0\mbox{ for all }P\in\fP_{sem}\}
\]
and consider the filtration $\F\vee \cN_{sem}=(\cF_t\vee \cN_{sem})_{t\geq0}$. Note that this is still much smaller than the augmentation with all $\fP_{sem}$-polar sets (or even the $P$-augmentation for some $P\in\fP_{sem}$), because we are only adding sets already included in $\cF$. In particular, all elements of $\cF_t\vee \cN_{sem}$ are Borel sets and an $\F\vee \cN_{sem}$-progressively measurable process is automatically $\cF\otimes \cB(\R_+)$-measurable. For the purposes of the present paper, both versions are sufficient.

\begin{proposition}\label{prop:quadvarunabP}
  (i)  There exists an $\F$-predictable, $\S^d_+$-valued process $C$ such that
  \begin{equation*}
    C=\langle X^{c,P}\rangle^{(P)} \ \ \ P\mbox{-a.s.}\ \ \mbox{for all} \ \ P \in \fP_{sem},
  \end{equation*}
  where $X^{c,P}$ denotes the continuous local martingale part of $X$ under $P$ and $\langle X^{c,P} \rangle^{(P)}$ is its predictable quadratic variation under $P$. In particular, the paths of $C$ are $P$-a.s.\ increasing and continuous for all $P \in \fP_{sem}$.

  (ii) There exists an $\F\vee \cN_{sem}$-predictable, $\S^d_+$-valued process $\bar{C}$ with continuous increasing paths such that
  \begin{equation*}
    \bar{C}=\langle X^{c,P}\rangle^{(P)} \ \ \ P\mbox{-a.s.}\ \ \mbox{for all} \ \ P \in \fP_{sem}.
  \end{equation*}
\end{proposition}

\begin{proof}
  We begin with (ii). As a first step, we show that there exists an $\F\vee \cN_{sem}$-optional process $[X]$ with values in $\S^d_+$, having all paths c\`adl\`ag and of finite variation, such that
  \begin{equation*}
    [X]=[X]^{(P)} \ \ \ P\mbox{-a.s.}\ \ \mbox{for all} \ \ P \in \fP_{sem},
  \end{equation*}
  where $[X]^{(P)}$ is the usual quadratic covariation process of $X$ under $P$. To this end, we first apply Bichteler's pathwise integration \cite[Theorem~7.14]{Bichteler.81}, see also~\cite{Karandikar.95} for the same result in modern notation, to $\int X^i_-\,dX^j$, for fixed $1\leq i,j\leq d$. This integration was also used in \cite{NutzSoner.10,SonerTouziZhang.2010bsde,SonerTouziZhang.2010dual} in the context of continuous martingales; however, we have to elaborate on the construction to find a Borel-measurable version.

  Define for each $n\geq1$ the sequence $\tau^n_0:=0$,
  \[
    \tau^n_{l+1}:=\inf\big\{t\geq \tau^n_l\,\big|\, |X^i_t-X^i_{\tau^n_l}|\geq 2^{-n}\mbox{ or }|X^i_{t-}-X^i_{\tau^n_l}|\geq 2^{-n}\big\},\quad l\geq0.
  \]
  Since $X$ is c\`adl\`ag, each $\tau^n_l$ is an $\F$-stopping time and $\lim_l \tau^n_l(\omega)=\infty$ holds for all $\omega\in\Omega$.
  In particular, the processes defined by
  \[
    I^n_t:= X^i_{\tau^n_k}\big(X^j_t-X^j_{\tau^n_k}\big) + \sum_{l=0}^{k-1} X^i_{\tau^n_l}\big(X^j_{\tau^n_{l+1}}-X^j_{\tau^n_l}\big)
    \quad\mbox{for}\quad\tau^n_k< t\leq\tau^n_{k+1},\quad k\geq0
  \]
  are $\F$-adapted and c\`adl\`ag, thus $\F$-optional. Finally, we define
  \[
    I_t(\omega):=\limsup_{n\to\infty} I^n_t(\omega);
  \]
   then $I$ is again $\F$-optional. Moreover, it is a consequence of the Burkholder-Davis-Gundy inequalities that
  \begin{equation}\label{eq:BichtelerConv}
    \sup_{0\leq t\leq N}\bigg| I^n_t-\sideset{^{(P)\hspace{-7pt}}}{}{\int_0^t}X^i_{s-}\,dX^j_s\bigg|\to 0\quad P\as,\quad N\geq1
  \end{equation}
  for each $P\in\fP_{sem}$, where the integral is the usual It\^o integral under $P$. For two c\`adl\`ag functions $f,g$ on $\R_+$, let
  \[
    d(f,g)=\sum_{N\geq1} 2^{-N} (1\wedge \|f-g\|_N),
  \]
  where $\|\cdot\|_N$ is the uniform norm on $[0,N]$. Then $d$ metrizes locally uniform convergence and a sequence of c\`adl\`ag functions is $d$-convergent if and only if it is $d$-Cauchy. Let
  \[
    G=\{\omega\in \Omega\,|\, I^n(\omega)\mbox{ is $d$-Cauchy}\}.
  \]
  It is elementary to see that $G\in\cF$. Since~\eqref{eq:BichtelerConv} implies that $P(G)=1$ for all $P\in\fP_{sem}$, we conclude that the complement of $G$ is in $\cN_{sem}$. On the other hand, we note that the $d$-limit of a sequence of c\`adl\`ag functions is necessarily c\`adl\`ag. Hence,
  \[
    J^{ij}:=I\1_{G}
  \]
  defines an $\F\vee\cN_{sem}$-optional process with c\`adl\`ag paths. Define the $\R^{d\times d}$-valued process $Q=(Q^{ij})$ by
  \[
    Q^{ij}:=X^iX^j - J^{ij}-J^{ji}.
  \]
  Then $Q^{ij}=X^iX^j - {}^{(P)\hspace{-5pt}}\int X^i_-\,dX^j - {}^{(P)\hspace{-5pt}}\int X^j_-\,dX^i = ([X]^{(P)})^{ij}$ holds $P$-a.s.\ for all $P\in\fP_{sem}$; this is simply the integration-by-parts formula for the It\^o integral. In particular, $Q$ has increasing paths in $\S^d_+$ $P$-a.s.\ for all $P\in\fP_{sem}$. Since $Q$ is c\`adl\`ag, the set $G'=\{\omega\in\Omega\,|\, Q(\omega)$ is increasing in $\S^d_+\}$ is $\cF$-measurable and we conclude that
  \[
    [X]:=Q\1_{G'}
  \]
  is an $\F\vee\cN_{sem}$-optional process having c\`adl\`ag, increasing paths and satisfying $[X]=[X]^{(P)}$ $P$-a.s.\ for all $P\in\fP_{sem}$.

  The second step is to construct $\bar{C}$ from $[X]$. Recall that a c\`adl\`ag function $f$ of finite variation can be (uniquely) decomposed into the sum of a continuous part $f^c$ and a discontinuous part $f^d$; namely,
  \begin{equation*}
    f^d_t:= \sum_{0\leq s \leq t} (f_s-f_{s-}), \ \ \ \ f^c_t:=f_t-f^d_t,
  \end{equation*}
  where $f_{0-}:=0$. Since all paths of $[X]$ are c\`adl\`ag and of finite variation, we can define $\bar{C}:=[X]^c$. Then $\bar{C}$ is $\F\vee \cN_{sem}$-optional (e.g., by \cite[Proposition~1.16, p.\,69]{JacodShiryaev.03}), $\bar{C}_0=0$ and all paths of $\bar{C}$ are increasing in $\S^d_+$ and continuous. Hence, $\bar{C}$ is also $\F\vee \cN_{sem}$-predictable.
  Let $P\in\fP_{sem}$ and recall (see \cite[Theorem~4.52, p.\,55]{JacodShiryaev.03}) that
  \begin{equation*}
    [X]^{(P)} =\langle X^{c,P}\rangle^{(P)}  + \sum_{0\leq s \leq \cdot} (\Delta X_s)^2\quad P\as
  \end{equation*}
  By the uniqueness of this decomposition, we have that $\bar{C}=\langle X^{c,P}\rangle^{(P)}$ $P$-a.s., showing that $\bar{C}$ is indeed a second characteristic of $X$ under $P$.

  For the $\F$-predictable version~(i), we construct $[X]$ as above but with $G=G'=\Omega$; then $[X]$ is $\F$-optional (instead of $\F\vee\cN_{sem}$-optional) while lacking the path properties. On the other hand, all paths of $X$ are c\`adl\`ag and hence the process
  \[
     C' := [X] - \sum_{0\leq s \leq \cdot} (\Delta X_s)^2
  \]
  is well-defined and $\F$-optional. Next, define $C'_0:=0$ and (componentwise)
  \[
    C''_t:=\limsup_{n\to\infty} C'_{t-1/n},\quad t>0;
  \]
  then $C''$ is $\F$-predictable. Finally, the process $C:=C''\1_{C''\in\S^d_+}$ has the required properties, because for given $P\in\fP_{sem}$ the paths of $C'$ are already continuous $P$-a.s.\ and thus $C=C'=C''=\langle X^{c,P}\rangle^{(P)}$ $P$-a.s.
\end{proof}

\section{Differential Characteristics}\label{se:diffCharacteristics}

In this section, we prove Theorem~\ref{th:diffCharactMbl} and its corollary. The conditions of Theorem~\ref{th:diffCharactMbl} (which are the ones of Theorem~\ref{th:charactMbl}) are in force.
We recall the set of semimartingale measures under which $X$ has absolutely continuous characteristics,
\[
  \fP^{ac}_{sem}=\big\{P\in \fP_{sem}\,\big|\, \mbox{$(B^P,C,\nu^P)\ll dt$, $P$-a.s.}\big\}.
\]

\begin{lemma}
  The set $\fP^{ac}_{sem}\subseteq \fP(\Omega)$ is measurable.
\end{lemma}

\begin{proof}
  Let $(B^P, C, \nu^P)$ and $A^P$ be as stated in Theorem~\ref{th:charactMbl}.
  For all $P \in \fP_{sem}$, let $R^P$ be the $[0,\infty]$-valued process
  \begin{equation*}
  R^P:= \sum_{1\leq i\leq d}\Var(B^{P,i}) + \sum_{1\leq i,j\leq d}\Var(C^{ij}) + |A^P|,
  \end{equation*}
  where the indices $i,j$ refer to the components of the $\R^d$- and $\R^{d\times d}$-valued processes $B^P$, $C$ and
  \[
    \Var(f)_t:=  \lim_{n \to \infty} \sum_{k=1}^{2^n}  \big|f_{kt/2^n}-f_{(k-1)t/2^n}\big|
  \]
  for any real function $f$ on $\R_+$. (If $f$ is right-continuous, this is indeed the total variation up to time $t$, as the notation suggests.) This definition and the properties stated in Theorem~\ref{th:charactMbl} imply that $(P,\omega,t)\mapsto R^P_t(\omega)$ is measurable and that for each $P \in \fP_{sem}$, $R^P$ is finite valued $P$-a.s.\ and has $P$-a.s.\ right-continuous paths. Moreover, we have $P$-a.s.\ that (componentwise)
  \[
    dA^P \ll dR^P, \quad dB^P \ll dR^P \quad  \mbox{and} \quad dC \ll dR^P.
  \]
  Let
  \begin{equation*}
  \varphi^{P,n}_t(\omega)
  := \sum_{k\geq0} 2^n \big(R^P_{(k+1)2^{-n}}(\omega) - R^P_{k2^{-n}}(\omega)\big)
  \,\1_{(k 2^{-n}, (k+1) 2^{-n}]}(t)
  \end{equation*}
  for all $(P,\omega,t) \in \fP_{sem}\times\Omega\times\R_+$ and
  \begin{equation*}
  \varphi^{P}_t(\omega):= \limsup_{n \to \infty} \varphi^{P,n}_t(\omega), \ \ \ \ (P,\omega,t) \in \fP_{sem}\times\Omega\times\R_+.
  \end{equation*}
  Clearly $(P,\omega,t)\mapsto \varphi^{P}_t(\omega)$ is measurable. Moreover, $\varphi^P$ is $P$-a.s.\ the density of the absolutely continuous part of $R^P$ with respect to the Lebesgue measure; cf.\ \cite[Theorem V.58, p.\,52]{DellacherieMeyer.82} and the subsequent remark. That is, there is a decomposition
  $
  R^P_t(\omega)= \int_0^t \varphi^{P}_s (\omega) \,ds + \psi^P_t(\omega)$, $t\in\R_+$
  for $P$-a.e.\ $\omega \in \Omega$, with a function $t\mapsto \psi^P_t(\omega)$ that is singular with respect to the Lebesgue measure.
  In particular, $\fP^{ac}_{sem}$ can be characterized as
  \begin{equation*}
    \fP^{ac}_{sem}=\big\{ P \in \fP_{sem} \, \big| \, E^P[\1_{G}(P,\cdot)] =1 \big\}
  \end{equation*}
  with the set
  \begin{equation*}
    G:= \bigg\{ (P,\omega) \in \fP_{sem} \times \Omega \, \bigg| \, R^P_t(\omega) = \int_0^t \varphi^{P}_s(\omega) \,ds \ \mbox{for all } t \in \Q_{+} \bigg\}.
  \end{equation*}
  As $G$ is measurable, we conclude by Lemma~\ref{le:BorelcondexpPP} that $\fP^{ac}_{sem}$ is measurable.
\end{proof}

Next, we prove the remaining statements of Theorem~\ref{th:diffCharactMbl}.

\begin{proof}[Proof of Theorem~\ref{th:diffCharactMbl}]
  Let $B^P,C,\nu^P, K^P, A^P$ be as in Theorem~\ref{th:charactMbl} and let $P\in \fP^{ac}_{sem}$. Let
  \[
    \widehat{A}^P_t:= \limsup_{n\to\infty} A^P_{(t-1/n)\vee 0};
  \]
  then $\widehat{A}^P$ is $\F_{-}$-adapted and hence $\F$-predictable. Moreover, since we know a priori that
  $A^P$ has continuous paths $P$-a.s., we have $\widehat{A}^{P}=A^P$ $P$-a.s. Consider
  \[
    \widetilde{a}^P_t:= \limsup_{n\to\infty} n (\widehat{A}^P_{t} - \widehat{A}^P_{(t-1/n)\vee 0}).
  \]
  If we define $a^P:=\widetilde{a}^P\1_{\R_+}(\widetilde{a}^P)$, then $(P,\omega,t)\mapsto a^P_t(\omega)$ is measurable and $a^P$ is an $\F$-predictable process for every $P\in\fP^{ac}_{sem}$. Moreover, since $A^P_t\ll dt$ $P$-a.s., we also have $a^P_t\,dt=dA_t^P$ $P$-a.s.
  We proceed similarly with $B^P$ and $C$ to define processes $b^P$ and $c$ with values in $\R^d$ and $\S^d_+$, respectively, having the properties stated in Theorem~\ref{th:diffCharactMbl}.

  Let $\widetilde{K}^P(\omega,t,dx)$ be the $\mathcal{B}(\fP_{sem}) \otimes \mathcal{P}$-measurable kernel
  from~\eqref{eq:defKtildeKernel} and let
  $\widetilde{F}^P_{\omega,t}(dx)$ be the  kernel on $\R^d$ given $\fP^{ac}_{sem}\times \Omega \times\R_+$
  defined by
  \begin{equation*}
    \widetilde{F}^P_{\omega,t}(dx):= \widetilde{K}^P(\omega,t,dx)\,a^P_t(\omega).
  \end{equation*}
  It follows from Fubini's theorem that $\widetilde{F}^P_{\omega,t}(dx) \in \cL$ holds $P\times dt$-a.e.\ for all $P \in \fP^{ac}_{sem}$. To make this hold everywhere, let
  \begin{equation*}
  G= \Big\{ (P,\omega,t) \in \fP^{ac}_{sem}\times \Omega \times \R_+ \, \Big|  \int_{\R^d} |x|^2\wedge 1 \, \widetilde{F}^P_{\omega,t}(dx)<\infty \mbox{ and } \widetilde{F}^P_{\omega,t}(\{0\})=0 \Big\}.
  \end{equation*}
  Then $G \in \B(\fP^{ac}_{sem})\otimes \cF\otimes B(\R_+)$ and the complements of its sections,
  \[
    G^{P}:=\big\{(\omega,t) \in \Omega \times \R_+ \, \big| \, (P,\omega,t) \notin G \big\},
  \]
  satisfy
  \[
    G^{P}\in\cP   \quad\mbox{and}\quad (P\otimes dt)(G^{P})=0,\quad P \in \fP^{ac}_{sem}.
  \]
  Thus, if we define the kernel $F^P_{\omega,t}(dx)$ on $\R^d$ given $\fP^{ac}_{sem}\times\Omega\times\R_+$
  by
  \begin{equation*}
    F^P_{\omega,t}(E):= \1_{G}(P,\omega,t)\,\widetilde{F}^P_{\omega,t}(E), \quad E \in \B(\R^d);
  \end{equation*}
  then $F^P_{\omega,t}(dx) \in \cL$ for all $(P,\omega,t) \in \fP^{ac}_{sem}\times \Omega \times\R_+$, while
  $(\omega,t)\mapsto F^P_{\omega,t}(dx)$ is a kernel on $(\R^d,\B(\R^d))$ given $(\Omega\times\R_+, \P)$ for all $P \in \fP^{ac}_{sem}$ and
  \[
    F^P_{\omega,t}(dx) \,dt = \widetilde{K}^P(\omega,t,dx)\,dA^P_t(\omega)=K^P(\omega,t,dx)\,dA^P_t(\omega)=\nu^P(\omega,dt,dx)
  \]
  $P$-a.s.\ for all $P \in \fP^{ac}_{sem}$. Moreover, $(P,\omega,t) \mapsto \int_{E} |x|^2 \wedge 1 \, F^P_{\omega,t}(dx)$ is measurable for any $E \in \B(\R^d)$. Thus, by Lemma \ref{le:kernlevy}, the map $(P,\omega,t) \mapsto F^P_{\omega,t}(dx)$ is measurable with respect to $\cB(\cL)$. Finally, it is clear from the construction that $(b^P,c,F^P)$ are indeed differential characteristics of $X$ under $P$ for all $P\in \fP^{ac}_{sem}$.
\end{proof}

It remains to prove the measurability of the sets $\fP^{ac}_{sem}(\Theta)$.

\begin{proof}[Proof of Corollary~\ref{co:PThetaMbl}]
  Let $\Theta\subseteq \R^d \times \S^d_+\times\cL$ be a Borel set and let $(b^P,c,F^P)$ be a measurable version of the differential characteristics for $P\in\fP^{ac}_{sem}$ as in Theorem~\ref{th:diffCharactMbl}; then
  \[
    G:=\big\{(P,\omega,t)\,\big|\,(b^P_t,c_t,F^P_t)(\omega)\notin \Theta\big\}\,\in\, \cB(\fP^{ac}_{sem})\otimes \cF \otimes \cB(\R_+).
  \]
  Thus, by Fubini's theorem, $G':=\{(P,\omega)\,|\,\int_0^\infty \1_G(P,\omega,t)\,dt=0\}$ is again measurable. Since $G'$ consists of all $(P,\omega)$ such that $(b^P_t,c_t,F^P_t)(\omega)\in \Theta$ holds $P\otimes dt$-a.e., we have
  \begin{equation*}
    \fP^{ac}_{sem}(\Theta)= \big\{ P \in \fP^{ac}_{sem} \, \big| \, E^P[\1_{G'}(P,\cdot)]=1 \big\},
  \end{equation*}
  and the set on the right-hand side is measurable due to Lemma~\ref{le:BorelcondexpPP}.
\end{proof}


\newcommand{\dummy}[1]{}

\end{document}